\documentclass[a4paper, 11pt,reqno]{amsart}
\usepackage{amsthm,amssymb,amsmath,graphicx,psfrag,enumitem,mathrsfs}
\usepackage{mathtools}
\usepackage{dsfont}
\usepackage{bm}
\usepackage{subcaption}
\usepackage[foot]{amsaddr}
\usepackage[british]{babel}
\usepackage[babel]{microtype}
\usepackage[margin=1in]{geometry}
\usepackage{verbatim}
\usepackage{enumitem, hyperref}

\makeatletter
\def\namedlabel#1#2{\begingroup
    #2%
    \def\@currentlabel{#2}%
    \phantomsection\label{#1}\endgroup
}
\makeatother

\usepackage[textsize=footnotesize]{todonotes}

\usepackage{tikz}

\usetikzlibrary{backgrounds}

\usepackage[numbers,sort]{natbib}

\makeatletter
\def\NAT@spacechar{~}% NEW
\makeatother

\usepackage{hyperref}
\usepackage[capitalise]{cleveref}
\hypersetup{colorlinks=true,
   citecolor=blue,
   filecolor=blue,
   linkcolor=blue,
   urlcolor=blue
}

\crefname{figure}{Figure}{Figures}
\Crefname{figure}{Figure}{Figures}

\allowdisplaybreaks

\newtheorem{definition}{Definition}[section]
\newtheorem{claim}[definition]{Claim}

\newtheorem{proposition}[definition]{Proposition}
\newtheorem{theorem}[definition]{Theorem}

\newtheorem{lemma}[definition]{Lemma}

\newtheorem{conjecture}[definition]{Conjecture}

\newtheorem{remark}[definition]{Remark}
\newenvironment{claimproof}{%
\let\origqed=\qedsymbol%
\renewcommand{\qedsymbol}{$\blacktriangleleft$}%
\begin{proof}}{\end{proof}\let\qedsymbol=\origqed}

\numberwithin{equation}{section}

\newcommand{\bigO}{\ensuremath{\mathcal{O}}}

\newcommand{\cupdot}{\mathbin{\dot\cup}}

\def\disc{{\mbox{\rm ex}}}
\def\ex{{\mbox{\rm ex}}}

\def\pn{{\mbox{\rm pn}}}

\newcommand{\COMMENT}[1]{}
\newcommand{\COMNEW}[1]{}
\title{Path decompositions of random directed graphs}

\author[A.~Espuny D\'iaz]{Alberto Espuny D\'iaz}
\address[Espuny D\'iaz]{Institut f\"ur Mathematik, Technische Universit\"at Ilmenau, 98684 Ilmenau, Germany.}
\email{alberto.espuny-diaz@tu-ilmenau.de}

\author[V.~Patel]{Viresh Patel}
\author[F.~Stroh]{Fabian Stroh}
\address[Patel, Stroh]{Korteweg-de Vries Institute, Universiteit van Amsterdam, 1090 GE Amsterdam, The Netherlands.}
\email{V.S.Patel@uva.nl, f.j.m.stroh@uva.nl}

\thanks{This project started during a research visit of the first author to the Korteweg-de Vries Institute, which was partially funded by an LMS Early Career Research Travel Grant (Ref ECR-1920-23) and the University of Birmingham.
A.~Espuny Díaz has also received partial funding from the European Research Council (ERC) under the European Union's Horizon 2020 research and innovation programme (grant agreement no.~786198) and the Carl Zeiss Foundation.
V.~Patel is supported by the Netherlands Organisation for Scientific Research (NWO) through the Gravitation Programme Networks (024.002.003).
F.~Stroh is supported by the NWO TOP grant (613.001.601).}

\date{\today}

\begin{document}

\begin{abstract}
We consider the problem of decomposing the edges of a directed graph into as few paths as possible.
There is a natural lower bound for the number of paths needed in an edge decomposition of a directed graph $D$ in terms of its degree sequence: this is given by the excess of $D$, which is the sum of $|d^+(v) - d^-(v)|/2$ over all vertices $v$ of $D$ (here $d^+(v)$ and $d^-(v)$ are, respectively, the out- and indegree of $v$). 

A conjecture due to Alspach, Mason and Pullman from 1976 states that this bound is correct for tournaments of even order.
The conjecture was recently resolved for large tournaments.
Here we investigate to what extent the conjecture holds for directed graphs in general.
In particular, we prove that the conjecture holds with high probability for the random directed graph $D_{n,p}$ for a large range of $p$ (thus proving that it holds for most directed graphs).
To be more precise, we define a deterministic class of directed graphs for which we can show the conjecture holds, and later show that the random digraph belongs to this class with high probability.
Our techniques involve absorption and flows.
\end{abstract}

\maketitle

\thispagestyle{empty}

\section{Introduction}

An area of extremal combinatorics that has seen a lot of activity recently is the study of decompositions of combinatorial structures.
The prototypical question in this area asks whether, for some given class $\mathcal{C}$ of graphs, directed graphs, or hypergraphs, the edge set of each $H \in \mathcal{C}$ can be decomposed into parts satisfying some given property.
When studying decompositions, one often wishes to minimise the number of parts; e.g., in the case of edge colourings, determining the chromatic index amounts to partitioning the edges of a graph into as few matchings as possible.
In this paper, we will be concerned with decomposing the edges of directed graphs into as few (directed) paths as possible.

Let $D$ be a directed graph (or digraph for short) with vertex set $V(D)$ and edge set $E(D)$.
A path decomposition of $D$ is a collection of paths $P_1, \ldots, P_k$ of $D$ whose edge sets $E(P_1), \ldots, E(P_k)$ partition $E(D)$.
Given any directed graph $D$, it is natural to ask what the minimum number of paths in a path decomposition of $D$ is. 
This is called the \emph{path number} of $D$ and is denoted $\pn(D)$.
A natural lower bound on $\pn(D)$ is obtained by examining the degree sequence of $D$.
For each vertex $v \in V(D)$, write $d^+_D(v)$ (resp.\ $d_D^-(v)$) for the number of edges exiting (resp.\ entering) $v$.
The \emph{excess} at vertex $v$ is defined to be $\ex_D(v) \coloneqq d_D^+(v) - d_D^-(v)$.
We note that, in any path decomposition of $D$, at least $|\ex_D(v)|$ paths must start (resp.\ end) at $v$ if $\ex_D(v) \geq 0$ (resp.\ $\ex_D(v) \leq 0$).  
Therefore, we have
\[
\pn(D) \geq \ex(D) \coloneqq \frac{1}{2}\sum_{v \in V(D)} |\ex_D(v)|,
\]
where $\ex(D)$ is called the \emph{excess} of $D$.
Any digraph for which equality holds above is called \emph{consistent}.
Clearly, not every digraph is consistent; in particular, any Eulerian digraph $D$ has excess $0$ and so cannot be consistent.

For the class of tournaments (that is, orientations of the complete graph), Alspach, Mason, and Pullman~\cite{AMP} conjectured that every tournament with an even number of vertices is consistent. 

\begin{conjecture}
\label{conj:Pull}
Every tournament\/ $T$ with an even number of vertices is consistent.
\end{conjecture}

Many cases of this conjecture were resolved by the second author together with Lo, Skokan, and Talbot~\cite{LPST20}, and the conjecture has very recently been completely resolved (for sufficiently large tournaments) by Gir{\~a}o, Granet, K{\"u}hn, Lo, and Osthus~\cite{GGKLO20}.
Both results relied on the robust expanders technique, developed by K{\"u}hn and Osthus with several coauthors, which has been instrumental in resolving several conjectures about edge decompositions of graphs and directed graphs; see, e.g., \cite{CKLOT,KO,KO2}.

The conjecture seems likely to hold for many digraphs other than tournaments: indeed, the conjecture was stated only for even tournaments probably because it considerably generalised the following conjecture of Kelly, which was wide open at the time.
Kelly's conjecture states that every regular tournament has a decomposition into Hamilton cycles.
The solution of Kelly's conjecture for sufficiently large tournaments was one of the first applications of the robust expanders technique \cite{KO}.

A natural question then arises from \cref{conj:Pull}: which directed graphs are consistent?
It is NP-complete to determine whether a digraph is consistent~\cite{deVos20}, and so we should not expect to characterise consistent digraphs.
Nonetheless, here we begin to address this question by showing that the large majority of digraphs are consistent.
We consider the random digraph $D_{n,p}$, which is constructed by taking $n$ isolated vertices and inserting each of the $n(n-1)$ possible directed edges independently with probability $p$.
Our main result is the following theorem.

\begin{theorem}
\label{thm:mainintro}
Let $\log^4n/n^{1/3}\leq p \leq1-\log^{5/2}n/n^{1/5}$. Then, asymptotically almost surely (a.a.s.~for short)\footnote{That is, with probability tending to $1$ as $n$ goes to infinity.}~$D_{n,p}$ is consistent.
\end{theorem}

Notice that some upper bound on $p$, as in the above theorem, is necessary because, when $p=1$, we have that $\ex(D_{n,p})=0$ (with probability $1$) and so $D_{n,p}$ cannot be consistent.
Moreover the property of being consistent is not a monotone property, that is, adding  edges to a consistent digraph does not imply the resulting digraph is consistent.
Therefore, unlike many other properties, we should not necessarily expect a threshold for the consistency of random digraphs. We believe that the theorem holds for much smaller (and larger) values of $p$, and perhaps even that no lower bound on $p$ is necessary.
For this reason, we have not tried to optimise the polylogarithmic terms in our bounds on $p$.

It is interesting to compare Theorem~\ref{thm:mainintro} to the situation of the chromatic index for random graphs.
Recall that $\chi'(G)$ (the chromatic index of $G$) is the minimum number of matchings in an edge decompostion of a graph $G$ into matchings, and that $\Delta(G)$ (the maximum degree of $G$) is an obvious lower bound for this number (just as $\ex(D)$ is an obvious lower bound for $\pn(D)$). 
Erd\H{o}s and Wilson~\cite{EW} showed that a.a.s.~the random graph $G=G_{n,p}$ satisfies $\chi'(G) = \Delta(G)$ for $p=1/2$.
Frieze, Jackson, McDiarmid, and Reed~\cite{FJMR99} extended this to all constant values of $p\in(0,1)$.
Recently, this was extended to all $p=o(n)$ by Haxell, Krivelevich, and Kronenberg~\cite{HKK}; this might suggest that Theorem~\ref{thm:mainintro} could also hold with no lower bound on $p$. 

The proof of \cref{thm:mainintro} does not use randomness in a very significant way.
In fact, we give a set of sufficient conditions for a digraph to be consistent and show that the random digraph (for suitable $p$) satisfies these conditions with high probability.
Here we give a simplified  version of our main deterministic result (see \cref{thm:main} for the full statement).
For a digraph $D$, a subset of vertices $S \subseteq V(D)$, and a vertex $v \in V(D)$, we write $e_D(v,S)$ (resp.\ $e_D(S,v)$) for the number of outneighbours (resp.\ inneighbours) of $v$ in $S$.

\begin{theorem}
\label{thm:determintro}
There exist constants $n_0$ and $c$ such that the following holds.
Let $D=(V,E)$ be a digraph on $n\geq n_0$ vertices.
Set $t\coloneqq c (n \log n)^{2/5}$ and let $A^+\coloneqq\{v\in V:\ex_D(v)\geq t\}$, $A^-\coloneqq\{v\in V:\ex_D(v)\leq -t\}$, and $A^0\coloneqq V\setminus(A^+\cup A^-)$.
Assume there is some $d \geq t$ such that
\begin{enumerate}[label=$(\mathrm{\roman*})$]
    \item\label{mainintroitem1} for every $v\in A^+$ we have $d/4\leq e_D(v,A^-)\leq d$,
    \item\label{mainintroitem2} for every $v\in A^-$ we have $d/4\leq e_D(A^+,v)\leq d$,
    \item\label{mainintroitem3} for every $v \in A^+ \cup A^-$ we have $e_D(v,A^0),e_D(A^0,v)\leq \min\{d/3, t^2/10^6\}$, and
    \item\label{mainintroitem4} for every $v\in A^0$ we have $e_D(A^+,v), e_D(v,A^-)\geq d/3$.
\end{enumerate}
Then, $D$ is consistent.
\end{theorem}

Here is a concrete class of examples to which \cref{thm:determintro} applies. 
Take the edge-disjoint union of $D=(V,E)$ and $D'=(V,E')$, where $D$ is any digraph obtained by taking a regular bipartite graph of degree $t \geq c (n\log n)^{2/5}$  and orienting all edges from one part to the other, and $D'$ is any Eulerian digraph of maximum degree at most $3t$.
One can easily check that \cref{thm:determintro}  applies to such digraphs (here $A^0$ is empty), and so such digraphs are consistent.
Note therefore that \cref{thm:determintro} can be applied to many digraphs that are far from having any expansion or pseudorandom properties; e.g., digraphs satisfying the conditions of \cref{thm:determintro} could easily be disconnected or weakly connected.

Broadly speaking, our proof relies on the use of the so-called absorption technique, an idea due to R{\"o}dl, Ruci{\'n}ski, and Szemer{\'e}di~\cite{RRS} (with special forms appearing in earlier work, e.g.,~\cite{Kriv}).
We adapt and refine some of the absorption ideas used in \cite{LPST20}, but we also require several new ingredients.
We explain the main ideas of our proof in the next section.
In contrast to the previous work on this question~\cite{LPST20,GGKLO20}, our proof does not make use of robust expanders.
Preliminary ideas for this work came from de Vos~\cite{deVos20}.

The rest of this paper is organised as follows.
We give a sketch of the proof of \cref{thm:mainintro} in \cref{sec:sketch}. 
\Cref{sec:prel} is dedicated to giving common definitions and citing results we use.
In \cref{sec:pathdecomp} we describe the \emph{absorbing structure} and we show how to use it to decompose directed graphs $D$ satisfying certain properties into $\ex(D)$ paths.
Finally, in \cref{sec:decomprandom} we show that the random digraph contains the absorbing structure and satisfies these properties with high probability.
The proof of \cref{thm:mainintro} appears in \cref{sec:decomprandom} and the proof of \cref{thm:determintro} appears in \cref{sec:pathdecomp}.

\section{Proof sketch}
\label{sec:sketch}

Let $D=D_{n,p}$ with $p$ as in \cref{thm:mainintro}.
We divide the vertices of $D$ into sets $A^+$, $A^-$ and $A^0$ depending on whether $\ex_D(v)\geq t$, $\ex_D(v)\leq -t$, or $-t < \ex_D(v) < t$, respectively, for a suitable choice of $t$ (as a function of $n$ and $p$).
One can show that, with high probability, $A^+$ and $A^-$ have roughly the same size and $A^0$ is small.

We start by setting aside an absorbing structure $\mathcal{A}$ which consists of a set of edge-disjoint (short) paths of $D$.
Each vertex $v \in V(D)$ will have a set of paths $f(v)$ from $\mathcal{A}$ assigned to it, where the sets $f(v)$ partition $\mathcal{A}$.
In particular, for each $v \in A^+$ (resp.\ $v \in A^-$), $f(v)$ consists of single-edge paths from $v$ to $A^-$ (resp.\  $A^+$ to $v$) and, for each $v \in A^0$, $f(v)$ consists of a path with two edges which goes from $A^+$ to $A^-$ through $v$.
We think of $\mathcal{A}$ interchangeably as a set of paths and as a digraph that is the union of those paths.
We will require that $|f(v)|$ is sufficiently large for every vertex $v$ but at the same time that $\ex_{\mathcal{A}}(v) \leq \ex_D(v)$ for every vertex $v$.   
We give a set of conditions that ensure the existence of one such absorbing structure in \cref{def:digraphclass1} (see \cref{lem:absstruct1,lem:absstruct2}), and \cref{sec:decomprandom} is devoted to showing, by using concentration inequalities for martingales, that $D_{n,p}$ fulfils these conditions (with high probability) for all values of $p$ in the desired range (and, in fact, for a slightly larger range than stated in \cref{thm:mainintro}).

% It turns out that it is useful to have a collection $\mathcal{A}$ of edges of $D$ where every edge of $\mathcal{A}$ is leaving $A^+$ or entering $A^-$ (or both).
% We think of these edges as ``attached'' to one of their endpoints.
% We make sure that each vertex has a sufficiently large amount of edges of $\mathcal{A}$ attached to it and, at the same time, no vertex $v\in A^+\cup A^-$ is incident to more than $|\ex_D(v)|$ of these edges.
% The edges of $\mathcal{A}$ attached to vertices in $A^+$ are all directed towards $A^-$, those attached to vertices in $A^-$ all come from $A^+$, and vertices in $A^0$ have an equal number of edges entering from $A^+$ and leaving to $A^-$ attached to them.
% We call the set $\mathcal{A}$ an \emph{absorbing structure}; its precise definition is given later, in \cref{def:absstruct}.
% We give a set of conditions that ensure the existence of one such absorbing structure in \cref{def:digraphclass1} (see \cref{lem:absstruct1,lem:absstruct2}), and \cref{sec:decomprandom} is devoted to showing, by using concentration inequalities for martingales, that $D_{n,p}$ fulfils these conditions (with high probability) for all values of $p$ in the desired range (and, in fact, for a slightly larger range than stated in \cref{thm:mainintro}).

Next it is straightforward to obtain a set of edge-disjoint paths $\mathcal{P}$ in $D \setminus E(\mathcal{A})$ such that $|\mathcal{P}| + |\mathcal{A}| = \ex(D)$, and such that, writing $D' \coloneqq D \setminus (E(\mathcal{A}) \cup E(\mathcal{P}))$, we have $\ex(D')=0$.
So $\mathcal{P} \cup \mathcal{A}$ gives the correct number (i.e., $\ex(D)$) of edge-disjoint paths but the edges in $D'$ are not covered, and moreover $D'$ is Eulerian.
Our goal now is to slowly combine edges of $\mathcal{A}$ with edges of $D'$ to create longer paths in such a way that we maintain exactly $\ex(D)$ paths at every stage (\emph{absorbing} the edges of $D'$).
If we manage to combine all the edges of $D'$ in this way, then we have decomposed $D$ into $\ex(D)$ paths, thus proving that $D$ is consistent.

% In order to obtain an optimal path decomposition of $D$, we first remove the edges of $\mathcal{A}$ from~$D$.
% Then, we greedily find a set of edge-disjoint paths $\mathcal{P}$, each of them starting from a vertex with positive excess and ending in a vertex with negative excess, such that the digraph $D'$ obtained from $D$ by removing the edges of $\mathcal{A}$ and $\mathcal{P}$ satisfies $\ex(D')=0$.
% We should think of $\mathcal{A}$ as a collection of short paths from $A^+$ to $A^-$, consisting of single edges, or two edges in the case of edges attached to vertices in $A^0$.
% Our choice of $\mathcal{A}$ and $\mathcal{P}$ means that, together, these are already $\ex(D)$ paths.
% Our goal now is to slowly combine edges of $\mathcal{A}$ with edges of $D'$ to create longer paths in such a way that we maintain exactly $\ex(D)$ paths at every stage (\emph{absorbing} the edges of $D'$).
% If we manage to combine all the edges of $D'$ in this way, then we have decomposed $D$ into $\ex(D)$ paths, thus proving that $D$ is consistent. 

To begin the process of absorption, we apply a recent result of \citet{KLMN21} (improving on an earlier result of \citet{HMSSY}) which allows us to decompose the edges of $D'$ into $O(n\log n)$ cycles.
The core idea then is to combine certain paths from $\mathcal{A}$ with each cycle $C$ given by the decomposition, and to decompose their union into paths; we refer to this as \emph{absorbing} the cycle.
Crucially, in order to keep the number of paths invariant, we will combine each cycle $C$ with a set $\mathcal{A}_C$ of two paths from $\mathcal{A}$ and decompose $C \cup \mathcal{A}_C$ into two paths, as illustrated in \cref{fig:AbsEx} (and thereafter, the edges $\mathcal{A}_C$ are no longer available for use in absorbing other cycles). 

\begin{figure}
	\centering
	\includegraphics[width=0.65\textwidth]{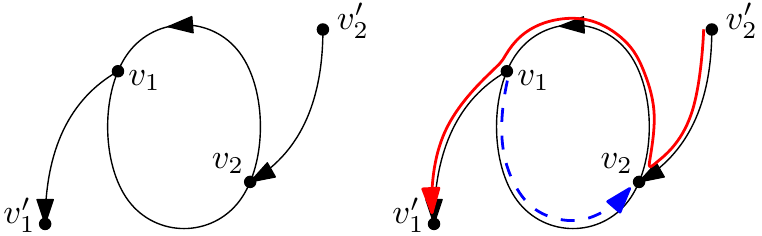}
	\caption{Left: One example of absorbing a cycle using two absorbing paths.
	We have $v_1,v_2$ on our cycle $C$ with $v_1\in A^+$, $v_2\in A^-$.
	We find paths $(v_1,v_1') \in f(v_1)$ and $(v_2',v_2) \in f(v_2)$ with $v_1'\in A^- \setminus V(C)$ and $v_2'\in A^+ \setminus V(C)$. \newline
	Right: The solid red and dashed blue lines show the two new paths $P_1\coloneqq v_2'v_2Cv_1v_1'$ and $P_2\coloneqq v_1Cv_2$, which use all involved edges. \newline
	Note that under certain circumstances, if $v_1',v_2'$ lie on $C$, we can still decompose all involved edges into two paths.}
	\label{fig:AbsEx}
\end{figure}

Therefore, we must allocate suitable absorbing paths to the cycles.
The two main challenges here are the following.
\begin{enumerate}[label=(\roman*)]
\item\label{sketchlabel1} The absorbing paths need to \emph{fit} the specific cycle, meaning they and the cycle can be decomposed into two paths.
Generally, given a cycle $C$, if we can find vertices $v_1, v_2 \in V(C)\setminus A^0$ and paths $P_1 \in f(v_1)$ and $P_2 \in f(v_2)$ where $P_1$ and $P_2$ have distinct endpoints not on $C$, then $P_1$ and $P_2$ will fit $C$ (see Figure~\ref{fig:AbsEx} for an example).
If both endpoints are on $C$, it is still sometimes possible (but not always) that $P_1$ and $P_2$ fit~$C$.
If $v_1$ or $v_2$ lie in $V(C)\cap A^0$, a similar idea can be used to find fitting paths.
\item\label{sketchlabel2} We only have a limited number of absorbing paths available at each vertex. 
\end{enumerate}
In order to address \ref{sketchlabel1}, we prepare more absorbing paths than we plan to use, as having the option to select from a sufficiently large number ensures that at least two fit a given cycle.
Any paths from $\mathcal{A}$ that we do not end up using to absorb a cycle remain as paths in the final decomposition.
In order to address point \ref{sketchlabel2}, we employ different strategies to assign absorbing edges to cycles, depending on the number of vertices that the cycle has in $A^+\cup A^-$. 
% This is because whether an absorbing edge works for a cycle is mainly dependent on whether both endpoints of the absorbing edge lie on the cycle, and for all absorbing edges we have that the endpoints they are not attached to lie in $A^+\cup A^-$. 
% We divide the cycles into three sets and treat them separately.

For cycles $C$ that are long (meaning they have many vertices in $A^+\cup A^-$), we greedily choose two paths that fit the cycle.
This is possible as each cycle contains a large number of vertices, so there are many choices for the possible absorbing paths, and we can always find two that fit the cycle.
Here, we allow both endpoints of the paths to be on $C$.

For cycles of medium length, we use a flow problem to assign vertices to cycles in such a way that each cycle is assigned a suitably large number of vertices dependent on its length, but such that no vertex is assigned to too many cycles.
It turns out that this choice of assignment means we can find two assigned vertices $v_1$ and $v_2$ per cycle and pick paths $P_i \in f(v_i)$ for $i=1,2$ that fit the cycle.
This strategy is wasteful in the sense that we sometimes assign more than two vertices to a cycle and thereby reserve more absorbing paths than we use.

For cycles that are short, it is easier to find fitting paths, as we are guaranteed to find absorbing paths that have their other endpoint off the cycle, as in the example in \cref{fig:AbsEx}.
However, it is harder to ensure that we do not use too many paths per vertex. 
In this case, we also use a flow problem to assign vertices to cycles, but we take multiple rounds and only decompose certain ``safe'' cycles in each round.
In addition, we absorb certain closed walks in each round, so we need to apply the result by Knierim et al.~between rounds in order to re-decompose the remaining edges into cycles, and this may generate new cycles which are long or of medium length.
Absorbing the short cycles is the most complicated process of the three, but it is the process we apply first so that the long and medium cycles that are produced as a byproduct can be absorbed by the appropriate processes described above.
It is also the only process in which we use the absorbing paths attached to vertices in $A^0$.

%%%%%%%%%%%%%%%%%%%%%%%%%%%%%%%%%%%%%%%%%%%%%%%%%%%%%%%%%%%%%%%%%%%%%
%%%%%%%%%%%%%%%%%%%%%%%%%%%%%%%%%%%%%%%%%%%%%%%%%%%%%%%%%%%%%%%%%%%%%

\section{Preliminaries}
\label{sec:prel}
\subsection{Basic definitions and notation}

For any $n\in\mathbb{Z}$, we will write $[n]\coloneqq\{i\in\mathbb{Z}:1\leq i\leq n\}$ and $[n]_0\coloneqq\{i\in\mathbb{Z}:0\leq i\leq n\}$.
Whenever we write $a=b\pm c$ for any $a,b,c\in\mathbb{R}$, we mean that $a\in[b-c,b+c]$.
Given any set $X$, $2^X$ denotes the set of all subsets of $X$.
Our logarithms are always natural logarithms.
We use the standard $\bigO$-notation for asymptotic statements, where the asymptotics will always be with respect to a parameter $n$.
Throughout, we ignore rounding whenever it does not affect our arguments.

In this paper, a \emph{digraph} $D=(V(D),E(D))$ is a loopless directed graph where, for each pair of distinct vertices $x,y\in V(D)$, we allow up to two edges between them, at most one in each direction.
We usually denote edges $(x,y)\in E(D)$ simply as $xy$.
The \emph{complement} of~$D$ is a digraph on the same vertex set as $D$ which contains exactly all the edges which are not contained in $D$.
Given any digraph $D$, we write $H\subseteq D$ to mean that $H$ is a \emph{subdigraph} of~$D$, that is, $V(H)\subseteq V(D)$ and $E(H)\subseteq E(D)$.
If $\mathcal{H}$ is a set of subdigraphs of $D$, we will sometimes abuse notation and treat $\mathcal{H}$ as the digraph obtained as the union of the digraphs which comprise~$\mathcal{H}$.
In particular, we will write $V(\mathcal{H})\coloneqq\bigcup_{H\in\mathcal{H}}V(H)$ and $E(\mathcal{H})\coloneqq\bigcup_{H\in\mathcal{H}}E(H)$.
Given any disjoint sets $A,B\subseteq V(D)$, we denote $E_D(A)\coloneqq\{ab\in E(D):a,b\in A\}$ and $E_D(A,B)\coloneqq\{ab\in E(D):a\in A,b\in B\}$.
If one of the sets consists of a single element (say, $A=\{a\}$), we will simplify the notation by setting $E(a,B)\coloneqq E(\{a\},B)$, and similarly for the rest of the notation.
We will write $e_D(A)\coloneqq|E_D(A)|$ and $e_D(A,B)\coloneqq|E_D(A,B)|$.
We denote $D[A]\coloneqq(A,E_D(A))$ for the subdigraph \emph{induced} by $A$ and, similarly, $D[A,B]\coloneqq(A\cup B,E_D(A,B))$ for the bipartite subdigraph \emph{induced} by $(A,B)$.
Given any $E\subseteq E(D)$, we write $D\setminus E\coloneqq(V(D),E(D)\setminus E)$.
Given any vertex $x\in V(D)$, we define its \emph{outneighbourhood} and \emph{inneighbourhood} as $N_D^+(x)\coloneqq\{y\in V(D):xy\in E(D)\}$ and $N_D^-(x)\coloneqq\{y\in V(D):yx\in E(D)\}$, respectively.
The \emph{outdegree} and \emph{indegree} of $x$ are given by $d_D^+(x)\coloneqq |N_D^+(x)|$ and $d_D^-(x)\coloneqq |N_D^-(x)|$, respectively.
Throughout, we may sometimes abuse notation by referring to a digraph by its edge set, especially in subscripts; the vertex set of such digraphs will always be clear from context.

As in the introduction, we define the \emph{excess} at $x$ to be $\ex_D(x)\coloneqq d_D^+(x)-d_D^-(x)$, and similarly define the \emph{positive excess} and \emph{negative excess} at~$x$ as $\ex_D^+(x)\coloneqq\max\{\ex_D(x),0\}$ and $\ex_D^-(x)\coloneqq\max\{-\ex_D(x),0\}$, respectively.
Observe that $\sum_{x\in V(D)}\ex_D(x)=0$.
We define the \emph{excess} of $D$ as 
\[\ex(D)\coloneqq\sum_{x\in V(D)}\ex_D^+(x)=\sum_{x\in V(D)}\ex_D^-(x)=\frac12\sum_{x\in V(D)}|\ex_D(x)|.\]

When we refer to paths, cycles, and walks in digraphs, we mean \emph{directed} paths, cycles, and walks, i.e., the edges are oriented consistently.
Given a digraph $D$, a \emph{walk} $W$ in $D$ is given by a sequence of (not necessarily distinct) vertices $W = v_1v_2 \cdots v_k$ where $v_1v_2, v_2v_3, \ldots, v_{k-1}v_k$ are distinct edges of $D$.
We also think of $W$ as being a subdigraph of $D$ with vertex set $\{v_1, \ldots, v_k\}$ and edge set $\{v_1v_2, v_2v_3, \ldots, v_{k-1}v_k\}$.
We also call $W$ a $(v_1, v_k)$-walk and sometimes denote it by $v_1Wv_k$ to emphasise that it starts at $v_1$ and ends at $v_k$, and we say $W$ is \emph{closed} if $v_1 = v_k$.
For two edge-disjoint walks $W_1=aW_1b = av_1\cdots v_kb$ and $W_2=bW_2c = bv_1'\cdots v_{\ell}'c$, we write $W_1W_2 =av_1\cdots v_kbv_1'\cdots v_{\ell}'c$ for the concatenation of $W_1$ and $W_2$.
This notation extends in the natural way for concatenating more than two walks.
For a walk $W=v_1\cdots v_k$, and $1 \leq i < j\leq k$, we write $v_iWv_j$ for the $(v_i,v_j)$-walk $v_iv_{i+1}\cdots v_j$ between $v_i$ and $v_j$.

In fact, we will mostly be concerned with paths and cycles rather than walks.
A walk $W = v_1 \cdots v_k$ is a \emph{path} if $v_1, \ldots, v_k$ are distinct vertices, and it is a \emph{cycle} if $v_1, \ldots, v_k$ are distinct except that $v_1 = v_k$.
The length of a walk, path, or cycle is the number of edges it contains.
We sometimes also consider degenerate single-vertex paths.
Note that, if $P_1$ is an $(a,b)$-path and $P_2$ is a $(b,c)$-path, where $P_1$ and $P_2$ are vertex-disjoint except at $b$, then $P_1P_2$ is an $(a,c)$-path.
For sets of vertices $X$ and $Y$, we say that a path $P$ is an $(X,Y)$-path if it starts in $X$ and ends in $Y$.

In this paper, we say a digraph $D$ is \emph{Eulerian} if $d^+_D(v)=d^-_D(v)$ for every $v\in V(D)$ or, equivalently, if $\ex(D)=0$.\footnote{This is different from the standard definition, which also asks that $D$ is strongly connected.}
A well-known consequence of this definition is the fact that the edge set of any Eulerian digraph can be decomposed into cycles.

We will sometimes need to consider a \emph{multidigraph} $D$, which is allowed to have multiple edges between any two vertices, in both directions (but it is still loopless).
Whenever $D$ is a multidigraph, all edge sets should be seen as multisets, while all vertex sets will remain simple sets. The notation and terminology above extend in the natural way to multidigraphs.

%%%%%%%%%%%%%%%%%%%%%%%%%%%%%%%%%%%%%%%%%%%%%%%%%%%%%%%%%%%%%%%%%%%%

\subsection{Path and cycle decompositions}

The following definitions are convenient.
\begin{definition}
A \emph{perfect decomposition} of a digraph $D$ is a set  $\mathcal{P} = \{P_1, \ldots, P_r \}$ of edge-disjoint paths of $D$ that together cover $E(D)$  with  $r = \disc(D)$.
(Thus, a digraph $D$ is consistent if and only if it has a perfect decomposition.)
\end{definition}

We will need the following basic facts. 

\begin{proposition}
\label{pr:PartDecomp1}
Let $D$ be a digraph with $\ex(D)>0$.
Then, there exists a path in $D$ from a vertex of positive excess to a vertex of negative excess.
\end{proposition}

\begin{proof}
First, repeatedly remove cycles from $D$ until this is no longer possible and call the resulting digraph $D'$; note that this does not affect the excess of any vertex.
Now any maximal path $P$ in $D'$ starts at a vertex that has no inneighbours (so it has positive excess) and ends at a vertex that has no outneighbours (so it has negative excess). 
\end{proof}

\begin{proposition}
\label{pr:PartDecomp2}
Suppose $D$ is a digraph, and let $X, Y \subseteq V(D)$ be disjoint.
If $P_1, \ldots, P_k$ are edge-disjoint $(X,Y)$-paths and $E(P_1) \cup \ldots \cup E(P_k) = E(D)$, then $\{P_1, \ldots, P_k\}$ is a perfect decomposition of $D$.
\end{proposition}

\begin{proof}
If we construct $D$ by adding the $k$ paths one at a time, we notice that the excess increases by one every time a path is added, so that $\ex(D) = k$.
\end{proof}

As mentioned in \cref{sec:sketch}, we will use ``absorbing structures'' (see \cref{def:absstruct}) to absorb Eulerian digraphs.
For this, we will first decompose the Eulerian digraphs into cycles.
We will use the following theorem of \citet{KLMN21} to achieve this.

\begin{theorem}
\label{th:compress}
There exists a constant $c'$ such that every Eulerian digraph $D$ on $n$ vertices can be decomposed into at most $c' n \log n$ edge-disjoint cycles.\footnote{In fact, the result of \citet{KLMN21} is slightly stronger, in the sense that $\log n$ can be replaced by $\log \Delta$, where $\Delta$ is the maximum (out- or in-)degree of $D$.}
\end{theorem}

%%%%%%%%%%%%%%%%%%%%%%%%%%%%%%%%%%%%%%%%%%%%%%%%%%%%%%%%%%%%%%%%%%%%%

\subsection{Flows}
\label{subsect:flows}
We recall some common definitions and facts about flow networks.
We note that flows are only used in the proofs of \cref{lem:absorbc2,lem:absorbc3}.

A \emph{flow network} is a tuple $(F,w,s,t)$, where $F=(V,E)$ is a digraph, $w\colon E\to \mathbb{R}$ is the \emph{capacity} function, and $s\in V$ is a \emph{source} (i.e., it only has outedges incident to it) and $t \in V$ is a \emph{sink} (i.e., it only has inedges incident to it).
A \emph{flow} for the flow network $(F,w,s,t)$ is a function $\phi\colon E\to \mathbb{R}^+$ such that, for all $e\in E$, we have $\phi(e)\leq w(e)$ and, for all $v\in V\setminus\{s,t\}$, we have $\sum_{u\in N^-_F(v)}\phi(uv)=\sum_{u\in N^+_F(v)}\phi(vu)$.
We define the \emph{value} of $\phi$ as $val(\phi)\coloneqq \sum_{v\in N^+_F(s)}\phi(sv)$.
A \emph{maximum flow} on a given flow network is a flow $\phi$ that maximises $val(\phi)$.

A partition $(U,W)$ of $V$ with $s\in U$, $t\in W$ is called a \emph{cut}, and we call the edge set $E_F(U,W)$ its corresponding \emph{cut-set}.
The \emph{capacity} $w((U,W))$ of a cut $(U,W)$ is the sum of the capacities of the edges of its cut-set, i.e., $w((U,W)) \coloneqq w(E_F(U,W))\coloneqq \sum_{e \in E_F(U,W)}w(e)$. 
A \emph{minimum cut} of the given flow network is a cut of minimum capacity. 
We make use of the following well-known theorem.

\begin{theorem}[Max-flow min-cut \cite{ford_fulkerson_1956}]
\label{thm:maxflow}
For every flow network with maximum flow $\phi$ and minimum cut $(U,W)$ we have that $\mathit{val}(\phi)=\mathit{w}((U,W))$.
\end{theorem} 

An easy fact states that, if all edge capacities are integers, then there exists a maximum flow such that all flow values are integers.

Given a flow $\phi$ on a flow network $(F,w,s,t)$, we define the \emph{residual digraph} $G_\phi$ of $G$ under $\phi$ as a directed graph with vertex set $V$ and edge set $\{uv\in E: \phi(uv)<w(uv)\}\cup\{vu: uv\in E, \; \phi(uv)>0\}$.
An $(s,t)$-path in a residual graph $G_\phi$ is called an \emph{augmenting path}, and it is easy to see that an augmenting path exists in $G_\phi$ if and only if $\phi$ is not a maximum flow.

%%%%%%%%%%%%%%%%%%%%%%%%%%%%%%%%%%%%%%%%%%%%%%%%%%%%%%%%%%%%%%%%%%%%%

\subsection{Random digraphs and probabilistic estimates}

In \cref{sec:decomprandom}, we begin working with random digraphs in the binomial model (although we also introduce slight variants of this model in the proofs of \cref{lem:absstruct1,lem:absstruct2}).
We denote by $D_{n,p}$ a random digraph on vertex set $[n]$ obtained by adding each of the possible $n^2-n$ edges with probability $p$, independently of all other edges.
Most of our results will be asymptotic in nature.
In particular, given a (di)graph property $\mathcal{P}$ and a sequence of random (di)graphs $\{G_i\}_{i>0}$ with $|V(G_i)|\to\infty$ as $i\to\infty$, we say that $G_i$ satisfies $\mathcal{P}$ \emph{asymptotically almost surely} (a.a.s.) if $\mathbb{P}[G_i\in\mathcal{P}]\to1$ as $i\to\infty$.

We will need to prove concentration results for different random variables.
For this, we will often use Chernoff bounds (see, e.g., the book of \citet[Corollary~2.3]{JLR}).

\begin{lemma}\label{lem:Chernoff}
Let $X$ be the sum of\/ $n$ mutually independent Bernoulli random variables, and let $\mu\coloneqq\mathbb{E}[X]$.
Then, for all $\delta\in(0,1)$ we have that $\mathbb{P}[X\geq(1+\delta)\mu]\leq e^{-\delta^2\mu/3}$ and $\mathbb{P}[X\leq(1-\delta)\mu]\leq e^{-\delta^2\mu/2}$.
In particular, $\mathbb{P}[|X-\mu|\geq\delta\mu]\leq2e^{-\delta^2\mu/3}$.
\end{lemma}

The following Chernoff-type bound extends \cref{lem:Chernoff} to allow us to deal with large deviations (see, e.g., the book of \citet[Theorem~A.1.12]{AS16}).

\begin{lemma}\label{lem:betaChernoff}
Let $X$ be the sum of\/ $n$ mutually independent Bernoulli random variables.
Let $\mu \coloneqq \mathbb{E}[X]$, and let $\beta>1$.
Then, $\mathbb{P}[X\geq\beta\mu]\leq\left(e/\beta\right)^{\beta\mu}$.
\end{lemma}

We will sometimes deal with random variables which are not independent, in which case we cannot obtain concentration results as above.
To deal with them, we will need the following version of the well-known Azuma-Hoeffding inequality (see, e.g., \cite[Theorem~2.25]{JLR}).
Given any sequence of random variables $X=(X_1,\ldots,X_n)$ taking values in a set $\Omega$ and a function $f\colon \Omega^n\to\mathbb{R}$, for each $i\in[n]_0$ define $Y_i\coloneqq\mathbb{E}[f(X)\mid X_1,\ldots,X_i]$.
The sequence $Y_0,\ldots,Y_n$ is called the \emph{Doob martingale} for $f$ and $X$.
All the martingales that appear in this paper will be of this form.

\begin{lemma}[Azuma's inequality]\label{lem:Azuma}
Let $Y_0,\ldots,Y_n$ be a martingale and suppose $|Y_i-Y_{i-1}|\leq c_i$ for all $i\in[n]$.
Then, for any $t>0$,
\[\mathbb{P}[|Y_n-Y_0|\geq t]\leq2\exp\left(\frac{-t^2}{2\sum_{i=1}^nc_i^2}\right).\]
\end{lemma}

We will also make use of the following well-known inequality; see, e.g., \cite[Theorem~368]{HLP}.

\begin{lemma}[rearrangement inequality]\label{lem:rearrangement}
Let $n\in\mathbb{N}$, and let $x_1\leq\ldots\leq x_n$ and $y_1\leq\ldots\leq y_n$ be real numbers.
Let $\sigma\in\mathfrak{S}_n$ be an arbitrary permutation.
Then,
\[\sum_{i=1}^nx_iy_{n+1-i}\leq\sum_{i=1}^nx_iy_{\sigma(i)}\leq\sum_{i=1}^nx_iy_i.\]
\end{lemma}

%%%%%%%%%%%%%%%%%%%%%%%%%%%%%%%%%%%%%%%%%%%%%%%%%%%%%%%%%%%%%%%%%%%%%
%%%%%%%%%%%%%%%%%%%%%%%%%%%%%%%%%%%%%%%%%%%%%%%%%%%%%%%%%%%%%%%%%%%%%

\section{Optimal path decompositions of digraphs}
\label{sec:pathdecomp}
In this section we give sufficient conditions for a digraph to be consistent.
These conditions will ensure that our digraph has a certain absorbing structure, and the absorbing structure will help us to decompose $D$ into $\ex(D)$ paths.

We begin by defining the classes of digraphs we will be working with throughout the rest of the paper.

\begin{definition}\label{def:digraphclass1}
Fix $p \in [0,1]$ and $0 \leq \lambda, \kappa \leq n$.
We say that $D=(V,E)$ is an \emph{$(n,p,\kappa,\lambda)$-digraph} if $|V|=n$ and the vertex set $V$ can be partitioned into three parts, $A^+$, $A^-$ and $A^0$ (where $A^0$ may be empty), in such a way that the following properties are satisfied:
\begin{enumerate}[label=$(\mathrm{P}\arabic*)$]
    \item\label{def:digraphitem1} For every $v\in A^+$ we have $\ex_D(v)\geq155\kappa$ and $np/4\leq e_D(v,A^-)\leq np$.
    \item\label{def:digraphitem2} For every $v\in A^-$ we have $\ex_D(v)\leq-155\kappa$ and $np/4\leq e_D(A^+,v)\leq np$.
    \item\label{def:digraphitem3} For every $v\in A^+\cup A^-$ we have $e_D(v,A^0),e_D(A^0,v)\leq \lambda$.
    \item\label{def:digraphitem4} For every $v\in A^0$ we have $e_D(A^+,v)\geq np/3$ and $e_D(v,A^-)\geq np/3$.
\end{enumerate}

We say that $D$ is an \emph{$(n,p,\kappa,\lambda)$-pseudorandom digraph} if it is an $(n,p,\kappa,\lambda)$-digraph and, additionally, the following property holds:
\begin{enumerate}[label=$(\mathrm{P}\arabic*)$,start=5]
    \item\label{def:digraphitem5} For every set $U\subseteq V$ with $|U|\geq\log n/(50p)$ we have $e_D(U)\leq100|U|^2p$.
\end{enumerate}
\end{definition}

Whenever we are given an $(n,p,\kappa,\lambda)$-digraph, we implicitly consider a partition of its vertex set into sets $A^+$, $A^-$ and $A^0$ which satisfy the properties described in \cref{def:digraphclass1}.
This partition is not necessarily unique; throughout this section, we simply assume that one such partition is given.
We will write $\dot{A}\coloneqq A^+\cup A^-$.

\begin{remark}\label{remark:newpseudorandom}
If $D$ is an $(n,p,\kappa,\lambda)$-(pseudorandom) digraph and $\kappa'\leq\kappa$ and $\lambda'\geq\lambda$, then $D$ is an $(n,p,\kappa',\lambda')$-(pseudorandom) digraph.
\end{remark}

We will see in \cref{sec:decomprandom} that a.a.s.~$D_{n,p}$ is an $(n,p,\kappa,\lambda)$-pseudorandom digraph, for a suitable choice of parameters. 
Our goal in this section is to prove the following theorem.

\begin{theorem}\label{thm:main}
There exists $n_0 \in \mathbb{N}$ with the following property. Suppose
 $n\in\mathbb{N}$, $p\in(0,1)$ and $\kappa,\lambda\in\mathbb{R}$ are parameters satisfying $n \geq n_0$ and
\begin{enumerate}[label=$(\mathrm{C}\arabic*)$]
     \item\label{thmitem1} $\kappa = 3N^{2/5}$,
    \item\label{thmitem2} $np \geq 365N^{2/5}$,  and
    \item\label{thmitem3} $\lambda = \min\{{np}/{3}, {\kappa^2}/{12}\}$,
\end{enumerate}
where $N\coloneqq c'n\log n$ and $c'$ is the constant from \cref{th:compress}.
Then, any $(n,p,\kappa,\lambda)$-digraph $D$ admits a perfect decomposition.

The same conclusion holds if $D$ is an $(n,p,\kappa,\lambda)$-pseudorandom digraph and \ref{thmitem1} and \ref{thmitem2} are replaced by
\begin{enumerate}[label=$(\mathrm{C}'\arabic*)$]
    \item\label{thmitem1'}  $\kappa = 6(N^2p)^{1/5}$, and
    \item\label{thmitem2'} $p \geq  n^{-1/3} \log^4 n$.
\end{enumerate}
\end{theorem}

Observe that, by \cref{remark:newpseudorandom}, we can extend \cref{thm:main} to any $(n,p,\kappa,\lambda)$-(pseudorandom) digraph where $\kappa$ is larger than the value given in \ref{thmitem1} or \ref{thmitem1'}, respectively, and $\lambda$ is smaller than the value given in \ref{thmitem3}.

We further remark that the constants in \cref{thm:main} as well as in \cref{def:digraphclass1} are not optimal.
In fact, there is a trade-off between some of them: by making one worse, others can be improved.
In order to ease readability, we refrain from stating the most general result possible, and simply note that a host of similar statements, with different constants, can be obtained by going through the proofs of the lemmas in this section.
Furthermore, we note that some of the conditions in \cref{def:digraphclass1} can be relaxed; in particular, \ref{def:digraphitem3} is only used in the proof of \cref{lem:absorbc2}, where only one of the two bounds stated in \ref{def:digraphitem3} is required.
Thus, as long as all vertices in $A^+ \cup A^-$ satisfy one (and the same) of the two bounds, \cref{thm:main} still holds, so it can be applied to a larger class of digraphs than stated in \cref{def:digraphclass1}.

Assuming \cref{thm:main}, we give the proof of \cref{thm:determintro}.

\begin{proof}[Proof of \cref{thm:determintro}]
We set $n_0$ as in \cref{thm:main} and $c\coloneqq 500 {(c')}^{2/5}$, where $c'$ is the constant from \cref{th:compress}.
Then, properties \ref{mainintroitem1}--\ref{mainintroitem4} of \cref{thm:determintro} and our choice of $A^+,A^-$ and $A^0$ correspond to \ref{def:digraphitem1}--\ref{def:digraphitem4} with $t\coloneqq c (n \log n)^{2/5}$, $d$, and $\min\{d/3,t^2/10^6\}$ playing the roles of $155\kappa$, $np$, and $\lambda$, respectively, so $D$ is an $(n,p,\kappa,\lambda)$-digraph.
By \cref{remark:newpseudorandom} and our choice of $t$ and $d$, we then conclude that $D$ is also an $(n,p,\kappa',\lambda')$-digraph which satisfies properties \ref{thmitem1}--\ref{thmitem3} of \cref{thm:main}\COMMENT{We have that $\kappa=t/155=500N^{2/5}/155>3N^{2/5}$, so \ref{thmitem1} holds.
\ref{thmitem2} holds since $d\geq t=500 N^{2/5}>365N^{2/5}$.
Finally,
\[\min\left\{\frac{d}{3},\frac{t^2}{10^6}\right\}=\min\left\{\frac{np}{3},\frac{500^2N^{4/5}}{10^6}\right\}\leq\min\left\{\frac{np}{3},\frac{9N^{4/5}}{12}\right\},\]
so \ref{thmitem3} holds too.}.
Thus, we may apply \cref{thm:main} and $D$ is consistent.
\end{proof}

\subsection{Finding absorbing structures}

The next definition describes the absorbing structure that we will find in $(n,p, \kappa, \lambda)$-digraphs $D$.
It will be used to absorb the majority of edges of $D$ into a set of $(A^+,A^-)$-paths that will end up being part of our perfect decomposition.
We will essentially show that, when we take an edge-disjoint union of our absorbing structure with any Eulerian subdigraph of $D$, the resulting digraph has a perfect decomposition.

\begin{definition}\label{def:absstruct}
Let $D$ be an $(n,p,\kappa,\lambda)$-digraph, and let $Z\subseteq V(D)$ and $t\in \mathbb{N}$.
A \emph{$(Z,t)$-absorbing structure} is a pair $\mathcal{A}=(E^{\mathrm{ab}},f)$, where $E^{\mathrm{ab}}\subseteq E(D)$ and $f\colon Z\to2^{E^{\mathrm{ab}}}$, such that
\begin{enumerate}[label=$(\mathrm{A}\arabic*)$]
    \item\label{def:absstruct1} if $z \in Z \cap A^+$, then $f(z)$ contains exactly $t$ edges from $E_D(z,A^-)$;
    \item\label{def:absstruct2} if $z \in Z \cap A^-$, then $f(z)$ contains exactly $t$ edges from $E_D(A^+,z)$;
    \item\label{def:absstruct3} if $z \in Z \cap A^0$, then $f(z)$ contains exactly $t$ edges from $E_D(A^+,z)$ and exactly $t$ edges from $E_D(A^+,z)$, and
    \item\label{def:absstruct4} the collection $\{f(z)\}_{z\in Z}$ is a partition of $E^{\mathrm{ab}}$; in particular, the sets $f(z)$ are disjoint.
\end{enumerate}
Note that, for convenience, for $z \in A^+ \cup A^-$, we often think of the $t$ edges in $f(z)$ as $t$ edge-disjoint $(A^+,A^-)$-paths of length $1$. For $z \in A^0$, we arbitrarily pair up the in- and outedges in $f(z)$ to create $t$ edge-disjoint $(A^+, A^-)$-paths of length $2$ through $z$.
\end{definition}

The following lemmas show the existence of absorbing structures in $(n,p, \kappa, \lambda)$ digraphs.

\begin{lemma}\label{lem:absstruct1}
Let $D$ be an $(n,p,\kappa,\lambda)$-digraph with $100\log n<\kappa\leq np/120$.
Then, $D$ contains an $(\dot{A},12\kappa)$-absorbing structure which contains at most $150\kappa$ edges incident to each $v\in \dot{A}$.
\end{lemma}

\begin{proof}
Consider $D[A^+,A^-]$.
We define a random subdigraph $D_q$ of $D[A^+,A^-]$ by including each of the edges of $E_D(A^+,A^-)$ with probability $q\coloneqq120\kappa/(np)$, independently of each other.
For each $v\in A^+$, let $\mathcal{B}_v$ be the event that $d^+_{D_q}(v)\notin[25\kappa,150\kappa]$.
Similarly, for each $v\in A^-$, let $\mathcal{B}_v$ be the event that $d^-_{D_q}(v)\notin[25\kappa,150\kappa]$.
By \ref{def:digraphitem1}, \ref{def:digraphitem2} and \cref{lem:Chernoff}, it follows that, for each $v\in \dot{A}$, $\mathbb{P}[\mathcal{B}_v]\leq e^{-\kappa/50}$\COMMENT{Assume $v\in A^+$ (the other case is proved analogously).
We apply \cref{lem:Chernoff} twice to obtain the two bounds on the resulting degree.
Note that $30\kappa\leq\mu\coloneqq\mathbb{E}[d^+_{D_q}(v)]\leq120\kappa$.
From the lower bound, by an application of \cref{lem:Chernoff} we have that
\[\mathbb{P}[d^+_{D_q}(v)<25\kappa]\leq\mathbb{P}[d^+_{D_q}(v)<25\mu/30]\leq e^{-\mu/72}\leq e^{-5\kappa/12}\]
(we apply \cref{lem:Chernoff} with $\delta=1/6$, and in the last bound we use the lower bound on $\mu$).
Similarly, from the upper bound we have that
\[\mathbb{P}[d^+_{D_q}(v)>150\kappa]\leq\mathbb{P}[d^+_{D_q}(v)>150\mu/120]\leq e^{-\mu/48}\leq e^{-5\kappa/8}\]
(we apply \cref{lem:Chernoff} with $\delta=1/4$).
The claim follows by adding these two terms; the bound we claim is very rough.
}.
Then, by a union bound over all $v\in \dot{A}$ and the lower bound on $\kappa$, we conclude that there exists a digraph $D'\subseteq D[A^+,A^-]$ such that, for each $v\in A^+$, $d^+_{D'}(v)\in[25\kappa,150\kappa]$, and for each $v\in A^-$, $d^-_{D'}(v)\in[25\kappa,150\kappa]$.

We are now going to randomly split the edges of $D'$ into two sets $E^+$ and $E^-$, and then prove that, with positive probability, $E^+$ contains an $(A^+,12\kappa)$-absorbing structure $\mathcal{A}^+$, and $E^-$ contains an $(A^-,12\kappa)$-absorbing structure $\mathcal{A}^-$.
It then immediately follows that $\mathcal{A}^+\cup\mathcal{A}^-$ is the desired $(\dot{A},12\kappa)$-absorbing structure.

For each $e\in E(D')$, with probability $1/2$ and independently of all other edges, we assign $e$ to $E^+$, and otherwise we assign it to $E^-$.
Let $D^+\coloneqq (\dot{A},E^+)$ and $D^-\coloneqq (\dot{A},E^-)$ (so, in particular, $D'=D^+\cup D^-$).
Now, for each $v\in A^+$, let $\mathcal{B}_v'$ be the event that $d^+_{D^+}(v)<12\kappa$, and for each $v\in A^-$, let $\mathcal{B}_v'$ be the event that $d^-_{D^-}(v)<12\kappa$.
In particular, by \cref{lem:Chernoff}, it follows that, for each $v\in \dot{A}$, $\mathbb{P}[\mathcal{B}_v']\leq e^{-\kappa/100}$\COMMENT{Assume $v\in A^+$ (the other case is analogous).
By the property we are assuming on $D'$ (i.e., all vertices have degree at least $25\kappa$), we have $\mu\coloneqq\mathbb{E}[d^+_{D^+}(v)]\geq25\kappa/2$.
Then, by \cref{lem:Chernoff},
\[\mathbb{P}[\mathcal{B}_v]=\mathbb{P}[d^+_{D^+}(v)<12\kappa]\leq\mathbb{P}[d^+_{D^+}(v)<24\mu/25]\leq e^{-\mu/2\cdot625}\leq e^{-\kappa/100}.\]
}.
By a union bound, we conclude that there exists a partition of $E(D')$ into $E^+$ and $E^-$ such that, for each $v\in A^+$, we have $d^+_{D^+}(v)\geq12\kappa$, and for each $v\in A^-$ we have $d^-_{D^-}(v)\geq12\kappa$\COMMENT{This is the event that none of the $\mathcal{B}_i$ occur, so it suffices to check that this happens with positive probability. 
By the union bound, the probability that any of the $\mathcal{B}_i$ occur is at most $ne^{-\kappa/100}<1$, so we are done.}.

In order to obtain the desired absorbing structure, for each $v\in A^+$ let $f(v)$ be an arbitrary set of $12\kappa$ of the edges of $E^+$ which contain $v$, and for each $v\in A^-$ let $f(v)$ be an arbitrary set of $12\kappa$ of the edges of $E^-$ which contain $v$.
\end{proof}

\begin{lemma}\label{lem:absstruct2}
Let $D$ be an $(n,p,\kappa,\lambda)$-digraph with $8\log(4n)<\kappa\leq np/12$\COMNEW{The lower bound is needed for the union bound, the upper bound is needed so that $q$ is well defined.}, $\lambda\leq np/3$ and $\kappa\lambda\geq 4np\log(2n)$.
Then, $D$ contains an $(A^0,3\kappa)$-absorbing structure which contains at most $5\kappa$ edges incident to each $v\in \dot{A}$.
\end{lemma}

\begin{proof}
Let $D'\coloneqq D[A^+,A^0]\cup D[A^0,A^-]$, and let $D_q$ be a random subdigraph of $D'$ obtained by adding each edge of $D'$ with probability $q\coloneqq12\kappa/(np)$ and independently of each other.
For each $v\in A^+$, let $\mathcal{B}_v$ be the event that $d^+_{D_q}(v)>5\kappa$.
Similarly, for each $v\in A^-$, let $\mathcal{B}_v$ be the event that $d^-_{D_q}(v)>5\kappa$.
Finally, for each $v\in A^0$, let $\mathcal{B}_v^+$ and $\mathcal{B}_v^-$ be the events that $d^-_{D_q}(v)<3\kappa$ and $d^+_{D_q}(v)<3\kappa$, respectively.

It follows from \ref{def:digraphitem3} and \cref{lem:Chernoff} that, for each $v\in \dot{A}$, we have $\mathbb{P}[\mathcal{B}_v]\leq e^{-{\kappa\lambda}/{(4np)}}$\COMMENT{Say $v\in A^+$ (the other case is analogous).
By \ref{def:digraphitem3}, we have $\mathbb{E}[d^+_{D_q}(v)]\leq\mu\coloneqq q\lambda\leq4\kappa$ (this follows by substituting the value of $q$ and using the bound on $\lambda$ from the statement).
Observe further that the variable $d^+_{D_q}(v)$, which is a binomial variable $\mathrm{Bin}(d^+_{D'}(v),q)$, is stochastically dominated by a variable $X\sim\mathrm{Bin}(\lambda,q)$ (this means that, for every $t$, $\mathbb{P}[d^+_{D_q}(v)\geq t]\leq\mathbb{P}[X\geq t]$).
Then, using this fact and \cref{lem:Chernoff},
\[\mathbb{P}[\mathcal{B}_v]=\mathbb{P}[d^+_{D_q}(v)>5\kappa]\leq\mathbb{P}[X>5\kappa]\leq\mathbb{P}[X>5\mu/4]\leq e^{-\mu/48}=e^{-\frac{1}{4}\frac{\kappa\lambda}{np}}.\]}.
Similarly, by \ref{def:digraphitem4} and \cref{lem:Chernoff}, for each $v\in A^0$ we have that $\mathbb{P}[\mathcal{B}_v^+],\mathbb{P}[\mathcal{B}_v^-]\leq e^{-\kappa/8}$\COMMENT{Consider $\mathcal{B}_v^+$ (the other case is analogous).
By \ref{def:digraphitem4} we have that $\mu\coloneqq\mathbb{E}[d^-_{D_q}(v)]\geq npq/3=4\kappa$.
Now, by \cref{lem:Chernoff},
\[\mathbb{P}[\mathcal{B}_v^+]=\mathbb{P}[d^-_{D_q}(v)<3\kappa]\leq\mathbb{P}[d^-_{D_q}(v)<3\mu/4]\leq e^{-\mu/32}\leq e^{-\kappa/8}.\]}.
By a union bound\COMMENT{The trivial bound is given by $2ne^{-\kappa/8}+ne^{-{\kappa\lambda}/{(4np)}}$, and this is $<1$ by the assumptions in the statement.}, we conclude that there exists $D^*\subseteq D'$ such that, for each $v\in A^+$, we have $d^+_{D_q}(v)\leq5\kappa$; for each $v\in A^-$, we have $d^-_{D_q}(v)\leq5\kappa$, and for each $v\in A^0$, we have $d^+_{D_q}(v),d^-_{D_q}(v)\geq3\kappa$.

In order to obtain the absorbing structure, for each $v\in A^0$, let $f(v)$ be the union of an arbitrary subset of $E_{D^*}(A^+,v)$ of size $3\kappa$ and an arbitrary subset of $E_{D^*}(v,A^-)$ of size $3\kappa$.
\end{proof}

\subsection{Using absorbing structures}

In this subsection, we show how to use absorbing structures to obtain perfect decompositions, and we use this to prove \cref{thm:main}.
As mentioned earlier, the idea will be to use these absorbing structures to absorb Eulerian digraphs.
The Eulerian digraphs will be decomposed into cycles, using \cref{th:compress}, and absorbed one cycle at a time.

Given an $(n,p,\kappa,\lambda)$-digraph $D$, we set $N\coloneqq c'n\log n$, where $c'$ is the constant given by \cref{th:compress}, so any Eulerian subdigraph of $D$ can be decomposed into at most $N$ cycles.
We call a cycle $C\subseteq D$ \emph{short} if $|V(C)\cap \dot{A}|\leq \kappa$, \emph{long} if $|V(C)\cap \dot{A}|\geq N/\kappa$, and \emph{medium} otherwise.
We will need a different strategy to absorb the set of cycles of each type.
We will show how to absorb long, medium and short cycles in \cref{lem:absorbc1,lem:absorbc2,lem:absorbc3}, respectively.

The following lemma shows how to absorb a single long or medium cycle, under suitable conditions, and will be used in \cref{lem:absorbc1,lem:absorbc2}.

\begin{lemma}
\label{lem:findabsedges}
Let $D$ be an $(n,p,\kappa,\lambda)$-digraph.
Let $C\subseteq D$ be a cycle with $\ell\coloneqq|V(C) \cap \dot{A}|>\kappa$ and $S\subseteq V(C)\cap \dot{A}$ with $|S|\geq\ell/\kappa+1$.
Let $\mathcal{A}=(E^{\mathrm{ab}},f)$ be an $(S,\kappa+2)$-absorbing structure such that $E(C)\cap E^{\mathrm{ab}}=\varnothing$.
Then, there exist distinct vertices $v_1,v_2\in S$ and edges $e_1\in f(v_1)$ and $e_2\in f(v_2)$ such that $E(C)\cup \{e_1,e_2\}$ can be decomposed into two $(A^+,A^-)$-paths.
\end{lemma}

\begin{proof}
Assume first that there are two distinct vertices $v_1,v_2\in S$ such that, for each $i\in[2]$, there is an edge $e_i\in f(v_i)$ whose other vertex is not contained in $V(C)$.
Observe that the definition of $\mathcal{A}$ ensures that $e_1\cup e_2$ is not a path of length $2$\COMMENT{All edges of $\mathcal{P}$ are oriented either towards $A^-$ or from $A^+$. By having the extra vertex outside $V(C)$, if $e_1$ and $e_2$ share a vertex, then they are both oriented towards this vertex or away from this vertex. So they do not form a path of length $2$.}.
Now, for each $i\in[2]$, if $e_i=v_ix_i$, let $P_i^+\coloneqq v_ix_i$ and $P_i^-\coloneqq v_i$, and if $e_i=x_iv_i$, let $P_i^+\coloneqq v_i$ and $P_i^-\coloneqq x_iv_i$.
Let $P$ be the $(v_1,v_2)$-subpath of $C$, and let $P'$ be the $(v_2,v_1)$-subpath of $C$.
The paths described in the statement are now given by $P_1\coloneqq P_1^-PP_2^+$ and $P_2\coloneqq P_2^-P'P_1^+$.
Since $e_1\cup e_2$ is not a path of length $2$, these two structures must indeed be paths and in all cases they are $(A^+,A^-)$-paths since the paths have the same start- and endpoints as $e_1$ and $e_2$.
See \cref{fig:absoutsideedges} for a visual representation of two of the four possible outcomes.
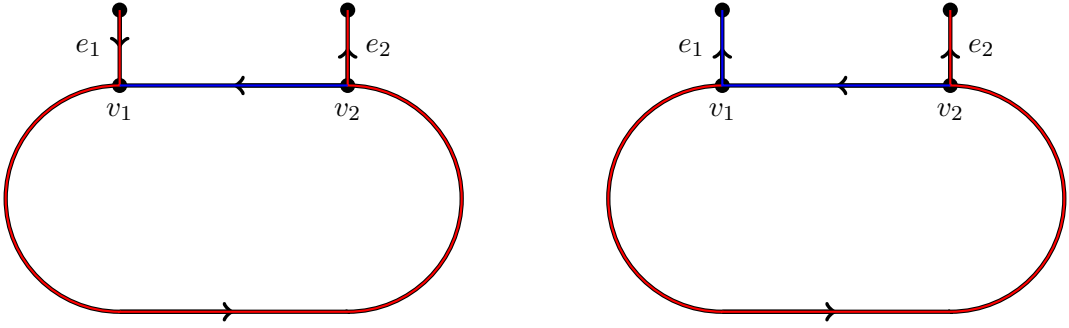
\begin{figure}
    \centering
    \begin{subfigure}[b]{0.49\textwidth}
        \centering
\begin{tikzpicture}
\draw[ultra thick, black] (0,0) -- (3,0) arc(90:-90:1.5) -- (3,-3) -- (0,-3) arc(270:90:1.5) -- cycle;
\draw[ultra thick, black,->] (3,0) -- (1.5,0);
\draw[ultra thick, black,->] (0,-3) -- (1.5,-3);
\draw (0,0) node[circle,fill,inner sep=2pt]{};
\draw (3,0) node[circle,fill,inner sep=2pt]{};
\node[anchor = north] at (0,-0.1){$v_1$};
\node[anchor = north] at (3,-0.1){$v_2$};
\draw[ultra thick, black] (3,0) -- (3,1);
\draw[ultra thick, black] (0,0) -- (0,1);
\draw[ultra thick, black,->] (3,0) -- (3,0.5);
\draw[ultra thick, black,->] (0,1) -- (0,0.5);
\draw (0,1) node[circle,fill,inner sep=2pt]{};
\draw (3,1) node[circle,fill,inner sep=2pt]{};
\node[anchor = east] at (-0.1,0.5){$e_1$};
\node[anchor = west] at (3.1,0.5){$e_2$};

\draw[thick, red] (0,1) -- (0,0) arc(90:270:1.5) -- (0,-3) -- (3,-3) arc(-90:90:1.5) -- (3,0) -- (3,1);
\draw[thick, blue] (3,0) -- (0,0);
\end{tikzpicture}
    \end{subfigure}
    \begin{subfigure}[b]{0.49\textwidth}
        \centering
\begin{tikzpicture}
\draw[ultra thick, black] (0,0) -- (3,0) arc(90:-90:1.5) -- (3,-3) -- (0,-3) arc(270:90:1.5) -- cycle;
\draw[ultra thick, black,->] (3,0) -- (1.5,0);
\draw[ultra thick, black,->] (0,-3) -- (1.5,-3);
\draw (0,0) node[circle,fill,inner sep=2pt]{};
\draw (3,0) node[circle,fill,inner sep=2pt]{};
\node[anchor = north] at (0,-0.1){$v_1$};
\node[anchor = north] at (3,-0.1){$v_2$};
\draw[ultra thick, black] (3,0) -- (3,1);
\draw[ultra thick, black] (0,0) -- (0,1);
\draw[ultra thick, black,->] (3,0) -- (3,0.5);
\draw[ultra thick, black,->] (0,0) -- (0,0.5);
\draw (0,1) node[circle,fill,inner sep=2pt]{};
\draw (3,1) node[circle,fill,inner sep=2pt]{};
\node[anchor = east] at (-0.1,0.5){$e_1$};
\node[anchor = west] at (3.1,0.5){$e_2$};

\draw[thick, red] (0,0) arc(90:270:1.5) -- (0,-3) -- (3,-3) arc(-90:90:1.5) -- (3,0) -- (3,1);
\draw[thick, blue] (3,0) -- (0,0) -- (0,1);
\end{tikzpicture}
    \end{subfigure}
    \caption{\footnotesize{A representation of the path decomposition of a cycle and two edges as proposed in \cref{lem:findabsedges}, in the case where we can find said edges with their endpoints outside $V(C)$.}}
    \label{fig:absoutsideedges}
\end{figure}

Therefore, we may assume that there are at least ${\ell}/{\kappa}>1$ vertices $v\in S$ such that all $e\in f(v)$ have both endpoints in $V(C)$.
Let us denote the set of these vertices by $S'$.
For each $v\in S'$, let $P_v$ be the shortest subpath of $C$ which does not contain $v$ and contains all other endpoints of the edges $e\in f(v)$ (recall that all said endpoints lie in $\dot{A}$).
In particular, $|V(P_v)\cap\dot{A}|\geq \kappa+2$.
Now label the vertices of $V(C)\cap\dot{A}$ as $y_1,\ldots,y_\ell$ in such a way that, when traversing $C$, they are visited in this (cyclic) order.
A simple counting argument shows the following.

\begin{claim}\label{claim:abs2edges}
There exist two distinct vertices $v_1,v_2\in S'$ such that $P_{v_1}$ and $P_{v_2}$ share at least two consecutive vertices of $V(C)\cap\dot{A}$.
\end{claim}

\begin{claimproof}
Assume the statement does not hold.
Then, any two paths from $\{P_v:v\in S'\}$ can intersect only at their endpoints, and any vertex of $V(C)\cap\dot{A}$ can be an endpoint of at most two paths.
This means that 
\[\sum_{v\in S'}|V(P_v)\cap\dot{A}|\leq\ell+|S'|.\]
However, using the bounds we have obtained so far, we can confirm that\COMMENT{We have
\[\sum_{v\in S'}|V(P_v)\cap\dot{A}|\geq|S'|(\kappa+2)=|S'|\kappa+2|S'|\geq(\ell/\kappa)\kappa+2|S'|=\ell+2|S'|>\ell+|S'|\]
(because $|S'|\geq\ell/\kappa>0$).}
\[\sum_{v\in S'}|V(P_v)\cap\dot{A}|\geq|S'|(\kappa+2)\geq\ell+2|S'|>\ell+|S'|.\qedhere\]
\end{claimproof}

By \cref{claim:abs2edges}, we can choose two edges $e_1\in f(v_1)$ and $e_2\in f(v_2)$ which form a ``crossing configuration'', that is, such that the vertices of $e_1$ and $e_2$ alternate when traversing $C$ (e.g., $wy$ and $zx$ are crossing edges in \cref{fig:abscrossingedges}).
In order to complete the proof, label the vertices of $e_1$ and $e_2$ as $w,x,y,z$ in such a way that, when traversing the cycle, they appear in this (cyclic) order and such that the edges are oriented towards $x$ and towards $y$, respectively (note that in any crossing configuration there exist two consecutive vertices into which the edges are directed).
The two paths of the statement are now given by $P_1\coloneqq wyCzx$ and $P_2\coloneqq zCy$, and these are $(A^+,A^-)$-paths since they have the same start- and endpoints as $e_1$ and $e_2$. See \cref{fig:abscrossingedges} for a visual representation.
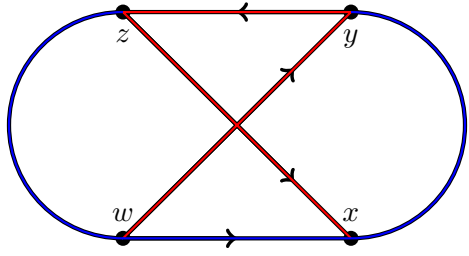
\begin{figure}
    \centering
\begin{tikzpicture}
\draw[ultra thick, black] (0,0) -- (3,0) arc(90:-90:1.5) -- (3,-3) -- (0,-3) arc(270:90:1.5) -- cycle;
\draw[ultra thick, black,->] (3,0) -- (1.5,0);
\draw[ultra thick, black,->] (0,-3) -- (1.5,-3);
\draw (0,0) node[circle,fill,inner sep=2pt]{};
\draw (3,0) node[circle,fill,inner sep=2pt]{};
\node[anchor = north] at (0,-0.1){$z$};
\node[anchor = north] at (3,-0.1){$y$};
\node[anchor = south] at (0,-2.9){$w$};
\node[anchor = south] at (3,-2.9){$x$};
\draw[ultra thick, black] (3,0) -- (0,-3);
\draw[ultra thick, black] (0,0) -- (3,-3);
\draw[ultra thick, black,->] (0,-3) -- (2.25,-0.75);
\draw[ultra thick, black,->] (0,0) -- (2.25,-2.25);
\draw (0,-3) node[circle,fill,inner sep=2pt]{};
\draw (3,-3) node[circle,fill,inner sep=2pt]{};

\draw[thick, red] (3,-3) -- (0,0) -- (3,0) -- (0,-3);
\draw[thick, blue] (0,0) arc(90:270:1.5) -- (0,-3) -- (3,-3) arc(-90:90:1.5) -- (3,0);
\end{tikzpicture}
    \caption{\footnotesize{A representation of the path decomposition of a cycle and two edges as proposed in \cref{lem:findabsedges}, in the case where we can find a ``crossing configuration''.}}
    \label{fig:abscrossingedges}
\end{figure}
\end{proof}

We now prove \cref{lem:absorbc1}, which shows how an absorbing structure can be used to absorb a collection of long cycles.
\COMNEW{In general, we want to make the constant of the absorbing structures as small as possible.}

\begin{lemma}
\label{lem:absorbc1}
Let $D$ be an $(n,p,\kappa,\lambda)$-digraph with $10\leq\kappa<N^{1/2}$\COMMENT{The upper bound is needed so that we may apply \cref{lem:findabsedges}. The lower bound ensures that some calculations go through. Note however that, for sufficiently large $n$, we may assume $\kappa$ is at least logarithmic, as otherwise $\mathcal{C}_1$ is empty.}.
Let $\mathcal{C}_1$ be a collection of edge-disjoint cycles in $D$ with $|\mathcal{C}_1|\leq 2N$ and such that, for each $C\in\mathcal{C}_1$, we have $|V(C)\cap\dot{A}|\geq N/\kappa$.
Let $\mathcal{A}=(E^{\mathrm{ab}},f)$ be an $(\dot{A},7\kappa -1)$-absorbing structure\COMNEW{If we want, the 7 can be improved to something slightly larger than 6.} with $E(\mathcal{C}_1)\cap E^{\mathrm{ab}}=\varnothing$.
Then, the digraph with edge set $E(\mathcal{C}_1)\cup E^{\mathrm{ab}}$ has a perfect decomposition in which each path is an $(A^+, A^-)$-path.
\end{lemma}

\begin{proof}
For each $C\in\mathcal{C}_1$, we are going to use \cref{lem:findabsedges} to find two edges $e_1,e_2 \in E^{\mathrm{ab}}$ such that $E(C) \cup \{e_1, e_2\}$ can be decomposed into two $(A^+,A^-)$-paths. 
We proceed iteratively as follows.

Assume that, for some of the cycles in $\mathcal{C}_1$, we have already found two edges as described above, and we now wish to do this for the next cycle $C \in \mathcal{C}_1$.
Let $\ell\coloneqq|V(C)\cap\dot{A}| \geq N/\kappa$.
We say that an edge $e\in E^{\mathrm{ab}}$ is \emph{available} if it has not been used to absorb any of the earlier cycles.
We say that a vertex $v\in V(C)\cap \dot{A}$ is \emph{available} if at least $\kappa+2$ edges of $f(v)$ are available, and we say that it is \emph{unavailable} otherwise.
Let $S_C\subseteq V(C)\cap \dot{A}$ be the set of available vertices.
Then, we can define an $(S_C,\kappa+2)$-absorbing structure $\mathcal{A}_C$ using edges from $E^{\mathrm{ab}}$ by selecting, for each $v\in S_C$, any set of $\kappa+2$ available edges from $f(v)$.
 
Note that the total number of edges assigned to cycles so far is at most $2|\mathcal{C}_1| \leq 4N$.
On the other hand, for each $v\in V(C)\cap \dot{A}$ which is unavailable, at least $5\kappa$\COMMENT{There are at most $\kappa+1$ available edges, so at least $7\kappa-1- \kappa-1=6\kappa-2\geq5\kappa$ have been assigned.} edges of $f(v)$ have already been assigned to cycles.
Therefore, the total number of unavailable vertices is at most $4N/(5\kappa)$, so $|S_C|\geq  \ell - 4N/(5\kappa) \geq \ell/5\geq\ell/\kappa+1$\COMMENT{By the upper bound on $\kappa$, we have that $\ell\geq N/\kappa>\kappa$, so in particular $\ell>10$.
The first inequality holds since $4N/(5\kappa)\leq4\ell/5$.
Now it suffices to check that $\ell/5\geq\ell/10+1$, which holds.}.
Therefore (noting that $\ell \geq N/\kappa > \kappa$), we can apply \cref{lem:findabsedges} (with $S_C$ and $\mathcal{A}_C$ playing the roles of $S$ and $\mathcal{A}$, respectively) to obtain two (available) edges $e_1,e_2 \in E^{\mathrm{ab}}$  such that $E(C)\cup\{e_1,e_2\}$ can be decomposed into two $(A^+,A^-)$-paths.
 
After each cycle has been handled in this way and, together with two edges, decomposed into two $(A^+,A^-)$-paths, we are left with some edges in $E^{\mathrm{ab}}$, which we treat as $(A^+,A^-)$-paths.
We therefore have a decomposition of $E(\mathcal{C}_1)\cup E^{\mathrm{ab}}$ into $(A^+,A^-)$-paths, which is a perfect decomposition by \cref{pr:PartDecomp2}.
\end{proof}

We will use flow problems in order to prove \cref{lem:absorbc2,lem:absorbc3}.
All our flow problems will follow a similar structure, so we introduce the following definition in addition to the common definitions given in \cref{subsect:flows}.

\begin{definition}\label{def:flow}
Let $D$ be a multidigraph and $\mathcal{C}$ be a set of edge-disjoint cycles of $D$.
Set $B\coloneqq V(\mathcal{C})$.
We define a flow network $(F,w,s,t)$ as follows.
We define a digraph $F=F(\mathcal{C})$ on vertex set $\{s\}\cupdot\mathcal{C}\cupdot B\cupdot\{t\}$, where $s$ and $t$ are the source and sink of the flow problem, respectively.
We set $E_1\coloneqq\{sC:C\in\mathcal{C}\}$, $E_2\coloneqq\{Cb:C\in\mathcal{C},b\in V(C)\}$, $E_3\coloneqq\{bt:b\in B\}$ and $E(F)\coloneqq E_1\cup E_2\cup E_3$.
Given any two functions $g\colon\mathcal{C}\to\mathbb{R}$ and $h\colon B\to\mathbb{R}$, we will write $\mathit{FP}(\mathcal{C};g,h)$ to denote the maximum flow problem on the digraph $F=F(\mathcal{C})$ defined above where each edge $sC\in E_1$ has capacity $w(sC)=g(C)$, each edge $Cb\in E_2$ has capacity $w(Cb)=1$, and each edge $bt\in E_3$ has capacity $w(bt)=h(b)$.
If $g$ or $h$ are constant functions, we will simply replace them by the corresponding constant in the notation.
\end{definition}

The following lemma shows how an absorbing structure can be used to absorb a collection of medium cycles.
\begin{lemma}
\label{lem:absorbc2}
Let $D$ be an $(n,p,\kappa,\lambda)$-digraph $D$ with $\kappa\geq\max\{12,(12\lambda)^{1/2},(72N^2)^{1/5}\}$, or an $(n,p,\kappa,\lambda)$-pseudorandom digraph $D$ with $\kappa\geq\max\{12,(12\lambda)^{1/2},(7200N^2p)^{1/5}, \sqrt{12/(25p)}\log n\}$\COMNEW{In order for \eqref{eq:absorbC2} to make sense, we also need to have $\kappa<N^{1/2}$.
But I think if we do not add this, then we can simply note that, if $\kappa\geq N^{1/2}$, then the statement is voidly true, as $\mathcal{C}_2$ is empty.}.
Let $\mathcal{C}_2$ be a collection of at most $2N$ edge-disjoint cycles in $D$ such that, for each $C\in\mathcal{C}_2$, we have
\begin{equation}\label{eq:absorbC2}
    \kappa<|V(C)\cap\dot{A}|<N/\kappa.
\end{equation}
Let $\mathcal{A}=(E^{\mathrm{ab}},f)$ be an $(\dot{A},2\kappa+1)$-absorbing structure with $E(\mathcal{C}_2)\cap E^{\mathrm{ab}}=\varnothing$.
Then, the digraph with edge set $E(\mathcal{C}_2)\cup E^{\mathrm{ab}}$ has a perfect decomposition in which each path is an $(A^+,A^-)$-path.
\end{lemma}

\begin{proof}
Given any digraph $H$ with $V(H)\subseteq V(D)$, let $g(H)\coloneqq\lceil{|V(H)\cap \dot{A}|}/\kappa\rceil+1$. 
We use a flow problem to assign, to each $C\in\mathcal{C}_2$, a set of $g(C)$
vertices of $V(C)\cap \dot{A}$ in such a way that no vertex is assigned to more than $\kappa$ cycles.
We will then use \cref{lem:findabsedges} to find two edges in $E^{\mathrm{ab}}$ with which to absorb $C$.
To this end, we construct a multiset of auxiliary cycles $\mathcal{C}_2'$ as follows.
We obtain $\mathcal{C}_2'$ from $\mathcal{C}_2$ by replacing each cycle $C\in\mathcal{C}_2$ by the auxiliary cycle $i(C)$ with vertices $V(C) \setminus A^0$ and whose cyclic vertex order is inherited from $C$.
Note that the cycles in $\mathcal{C}_2'$ are not necessarily cycles of $D$ and, indeed, the set $E(\mathcal{C}_2')$ (which forms a multidigraph) includes all the edges of $E(\mathcal{C}_2)$ inside $\dot{A}$ as well as an extra edge every time a cycle in $\mathcal{C}_2$ leaves and reenters $\dot{A}$.
We note for later that, since $e_D(v, A^0) \leq \lambda$ by \ref{def:digraphitem3}, the number of these extra edges contained in any $T \subseteq \dot{A}$ is at most $\lambda |T|$.
Consider $\mathit{FP}(\mathcal{C}'_2;g,\kappa)$.

\begin{claim}
\label{cl:flowc2}
$\mathit{FP}(\mathcal{C}'_2;g,\kappa)$ has a flow $\phi$ with $\mathit{val}(\phi)=\sum_{C\in \mathcal{C}'_2}g(C)$.
\end{claim}

\begin{claimproof}
Throughout this proof we use the notation set up in \cref{def:flow} and \cref{subsect:flows}.
As $M_0\coloneqq\{sC: C\in \mathcal{C}'_2\}$ is the cut-set of a cut of $F=F(\mathcal{C}_2')$ of capacity $\sum_{C\in \mathcal{C}'_2}g(C)$, by \cref{thm:maxflow} it remains to show that this is a minimum cut.
We assume the existence of a cut-set $M$ of $F$ with smaller capacity and will show that this contradicts our assumption on the value of $\kappa$.
Let $T\subseteq V(\mathcal{C}_2') \subseteq \dot{A}$ be the set of vertices that are separated from $t$ by $M$ and $T'\coloneqq V(\mathcal{C}_2') \setminus T$.
Let $S\subseteq \mathcal{C}'_2$ be the set of cycles which are not separated from $s$ by $M$, and $S'\coloneqq\mathcal{C}_2'\setminus S$.
These sets are illustrated in \cref{fig:lemC2}.
\begin{figure}
	\centering
	\includegraphics[width=0.45\textwidth]{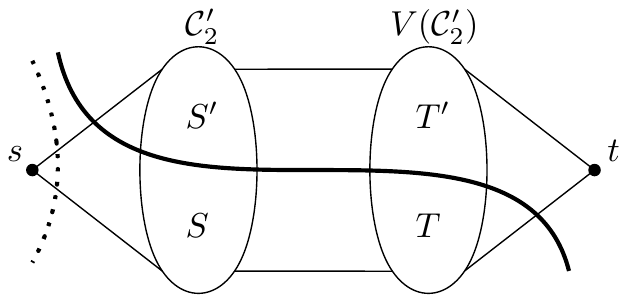}
	\caption{The graph $F(\mathcal{C}'_2)$. The thick dotted line illustrates the cut-set $M_0$. The regular thick line illustrates the cut-set $M$.}
	\label{fig:lemC2}
\end{figure}
Let $D_S$ be the multidigraph that is the union of the cycles in $S$.
We have that
\[w(M)=\sum_{C\in S'}g(C)+e_F(S,T')+|T|\kappa<\sum_{C\in \mathcal{C}_2'}g(C)=w(M_0),\]
which is equivalent to 
\begin{equation}\label{eq:flowC2eq1}
   \sum_{C\in S}g(C) > e_F(S,T')+|T|\kappa.
\end{equation}
(Note that we may assume $T\neq\varnothing$, as otherwise \eqref{eq:flowC2eq1} cannot hold\COMMENT{Assume $T=\varnothing$, so $T'=V(\mathcal{C}_2')$.
Then, for each $C\in S$, $e_F(\{C\},T')=|V(C)|=|V(C)\cap\dot{A}|>g(C)$.}.)
Now observe that\COMMENT{The term $e(D_S)/\kappa$ appears by ignoring the ceilings, and just looking at the definition of $D_S$.
Because of the ceilings, then, we could be adding much more, but at most $|S|$ more (and the final $|S|$ is by adding $1$ each time).}
\begin{equation}\label{eq:flowC2eq3}
\sum_{C\in S}g(C) = 
    \sum_{C\in S}\left(\left\lceil\frac{|V(C)|}{\kappa}\right\rceil+1\right)<\frac{e(D_S)}{\kappa}+2|S|.
\end{equation}
By \eqref{eq:absorbC2}, we have $|V(C)| > \kappa$ for all $C \in \mathcal{C}_2'$, so
it follows that
\begin{equation}\label{eq:flowC2eq2}
    |S|<\sum_{C\in S}|V(C)|/\kappa=e(D_S)/\kappa.
\end{equation}
Combining \eqref{eq:flowC2eq1}, \eqref{eq:flowC2eq3} and \eqref{eq:flowC2eq2}, it follows that
\begin{equation}
    \label{eq:lemc2:1}
    \frac{3e(D_S)}{\kappa} > e_F(S,T')+|T|\kappa.
\end{equation}
Next, since $|V(C)| \leq N/\kappa$ for all $C \in \mathcal{C}_2'$ by \eqref{eq:absorbC2} and $|\mathcal{C}_2'| = |\mathcal{C}_2| \leq 2N$, we have 
\begin{equation}
    \label{eq:lemc2:2}
    e(D_S)<2N^2/\kappa,
\end{equation}
Furthermore, since $e(D_S)=e_{D_S}(T)+e_{D_S}(T')+e_{D_S}(T',T)+e_{D_S}(T,T')$, we have\COMMENT{Recall $T$ and $T'$ are just sets of vertices.
Each edge between a vertex in $T'$ and a cycle in $S$ corresponds simply to this vertex lying on the cycle, which means it contributes to two edges in $D_S$ (one towards the vertex, and one away from it).
By considering all the cycles it belongs to, we recover precisely the indegree and the outdegree of that vertex in $D_S$. 
This gives the first equality.
Now, $\sum_{v\in T'}d^+_{D_S}(v)+d^-_{D_S}(v)=2e_{D_S}(T')+e_{D_S}(T',T)+e_{D_S}(T,T')$, so the second inequality is trivial.}
\begin{align*}
    e_F(S,T')&=\sum_{v\in T'}\frac{1}{2}(d^+_{D_S}(v)+d^-_{D_S}(v))\\
    &\geq \frac{1}{2}(e_{D_S}(T')+e_{D_S}(T',T)+e_{D_S}(T,T'))=\frac{1}{2}(e({D_S})-e_{D_S}(T)).
\end{align*}
Combining this with \eqref{eq:lemc2:1}, we have
\begin{equation*}
     \frac{6e({D_S})}{\kappa}>2e_F(S,T') \geq e({D_S})-e_{D_S}(T),
\end{equation*}
which implies 
\begin{equation}
\label{eq:lemc2:4}
    e_{D_S}(T)\geq \left(1-\frac{6}{\kappa}\right)e({D_S})\geq \frac{1}{2}e({D_S}).
\end{equation}

By the discussion before the claim concerning the construction of $\mathcal{C}_2'$, we have $e_{D_S}(T)\leq e_D(T)+\lambda|T|$.
This implies that either $e_{D_S}(T)\leq 2e_D(T)$ or $e_{D_S}(T)\leq 2\lambda|T|$.
If $e_{D_S}(T)\leq 2\lambda|T|$, then using \eqref{eq:lemc2:4} we obtain that $|T|\geq e(D_S)/(4\lambda)$, and combining this with \eqref{eq:lemc2:1} we have
\begin{equation*}
    \frac{3e(D_S)}{\kappa}>|T|\kappa\geq \frac{\kappa e(D_S)}{4\lambda},
\end{equation*}
so that $\kappa^2 < 12 \lambda$, contradicting our choice of $\kappa$. 
Therefore, we may assume 
\begin{equation}
\label{eq:DST}
e_{D_S}(T)\leq 2e_D(T). 
\end{equation}
Now we distinguish between the two cases in the statement of the lemma, i.e., when $D$ is an $(n,p,\kappa,\lambda)$-digraph and when $D$ is an $(n,p,\kappa,\lambda)$-pseudorandom digraph.

\noindent
\textbf{Case 1:} $D$ is an $(n,p,\kappa,\lambda)$-digraph. 
By \eqref{eq:DST} we have $e_{D_S}(T)\leq 2e_D(T) \leq2|T|^2$. 
Combined with \eqref{eq:lemc2:4}, we conclude that $|T|\geq \sqrt{e(D_S)/4}$.
By \eqref{eq:lemc2:1}, we have
\begin{equation*}
    \frac{3e(D_S)}{\kappa}>|T|\kappa\geq \kappa\sqrt{e(D_S)/4}
\end{equation*}
Combining this with \eqref{eq:lemc2:2}, we obtain that
\begin{equation*}
    2N^2/\kappa > e(D_S) > \kappa^4/36,
\end{equation*}
contradicting our choice of $\kappa \geq (72N^2)^{1/5}$.

\noindent
\textbf{Case 2:} $D$ is an $(n,p,\kappa,\lambda)$-pseudorandom digraph.
We further split this into two cases.
Assume first that $|T|\geq\log n/(50p)$, so by \ref{def:digraphitem5} and \eqref{eq:DST} we have that $e_{D_S}(T)\leq2e_D(T)\leq200|T|^2p$.
Combined with \eqref{eq:lemc2:4}, we have that $|T|\geq \sqrt{e(D_S)/(400p)}$. By \eqref{eq:lemc2:1}, we have
\begin{equation*}
    \frac{3e(D_S)}{\kappa} > |T|\kappa\geq \kappa\sqrt{e(D_S)/(400p)}
\end{equation*}
Combining this with \eqref{eq:lemc2:2}, we obtain that
\begin{equation*}
    2N^2/\kappa > e(D_S) > \kappa^4/(3600p),
\end{equation*}
contradicting our choice of $\kappa\geq(7200N^2p)^{1/5}$.

We may thus assume that $|T|<\log n/(50p)$.
In this case, we may consider any superset of $T$ of size $\log n/(50p)$ and, by applying \ref{def:digraphitem5} to this superset and considering \eqref{eq:DST}, we have that $e_{D_S}(T) \leq 2e_D(T) \leq 2\log^2n/(25p)$.
Then, by \eqref{eq:lemc2:4}, 
\[e(D_S)\leq 4\log^2n/(25p).\]
Now, using \eqref{eq:lemc2:1} and the fact that $T\neq\varnothing$, we also have that 
\[e(D_S)>|T|\kappa^2/3 \geq \kappa^2/3.\]
But these two bounds on $e(D_S)$ lead to a contradiction on our choice of $\kappa\geq\sqrt{12/(25p)}\log n$.
\end{claimproof}

We interpret the flow given by \cref{cl:flowc2} as follows.
As all capacities are integers, there exists an integer flow with value $\sum_{C\in \mathcal{C}'_2}g(C)$, so assume $\phi$ is such an integer flow.
For each cycle $C\in\mathcal{C}_2$, writing $C' = i(C)$, let $V_C\coloneqq\{v\in V(C') \subseteq V(C): \phi(C'v)=1\}$ be the vertices assigned to $C$.
As $\phi$ saturates all edges $sC'$, we have $|V_C|=g(C')=g(C)$.
The capacity $\kappa$ of the edges $vt$ with $v\in \dot{A}$ ensures that no vertex is assigned to more than $\kappa$ cycles of $\mathcal{C}_2$.

We will now iteratively assign two edges $e_1, e_2$ to each cycle $C\in \mathcal{C}_2$ so that $E(C) \cup \{e_1, e_2\}$ can be decomposed into two $(A^+,A^-)$-paths, where $e_1 \in f(v_1)$, $e_2 \in f(v_2)$ and $v_1, v_2 \in V_C$.
We do this as follows using \cref{lem:findabsedges}.
Assume that, for some of the cycles in $\mathcal{C}_2$, we have already found two edges as described above, and assume that we next want to do this for $C\in\mathcal{C}_2$.
We say that an edge $e\in E^{\mathrm{ab}}$ is \emph{available} if it has not been assigned to any of the previous cycles.
Then, for each $v\in V_C$, the number of edges $e\in f(v)$ that are available is at least $\kappa+2$ (since no vertex is assigned to more than $\kappa$ cycles and $\mathcal{A}=(E^{\mathrm{ab}},f)$ is an $(\dot{A}, 2\kappa + 1)$-absorbing structure).
Thus, we may define a $(V_C,\kappa+2)$-absorbing structure $\mathcal{A}_C$ using available edges from $E^{\mathrm{ab}}$ by selecting, for each $v \in V_C$, any set of $\kappa+2$ available edges at $v$.
Then, with $V_C$ and $\mathcal{A}_C$ playing the roles of $S$ and $\mathcal{A}$, respectively, \cref{lem:findabsedges} gives two edges $e_1\in f(v_1)$ and $e_2 \in f(v_2)$ with $v_1, v_2 \in V_C$ such that $E(C)\cup \{e_1,e_2\}$ can be decomposed into two $(A^+,A^-)$-paths.  
After repeating this for every cycle $C \in \mathcal{C}_2$ and treating each of the remaining edges in $E^{\mathrm{ab}}$ as an $(A^+,A^-)$-path, we have an edge decomposition of  $E(\mathcal{C}_2)\cup E^{\mathrm{ab}}$ into $(A^+,A^-)$-paths, which is a perfect decomposition by \cref{pr:PartDecomp2}.
\end{proof}

We have seen earlier in \cref{lem:findabsedges} how a single long or medium cycle can be absorbed using an absorbing structure. 
The following lemma shows how to absorb a single short cycle using our absorbing structure.
In fact, it is slightly more general: it shows how to absorb a short Eulerian digraph, namely one that is the union of two short edge-disjoint paths.
Another difference is that now we must work with vertices in $A^0$.
As before, in order to absorb a cycle $C$, we take two suitable vertices $v_1,v_2 \in V(C)$.
For long and medium cycles, both $v_1$ and $v_2$ had been in $\dot{A}$, and we used a single edge in $f(v_1)$ and a single edge in $f(v_2)$ for absorption.
For short cycles, if $v_1,v_2\in\dot{A}$, we do the same, but here one or both may be in $A^0$.
If, for instance, $v_1\in A^0$, we use a pair of edges from $f(v_1)$ (which should be thought of as an $(A^+,A^-)$-path of length two through~$v_1$) for absorption.  

\begin{lemma}
\label{lem:findleavingedges}
Let $D$ be a $(n,p,\kappa,\lambda)$-digraph and $v_1,v_2\in V(D)$.
Let $P_1\subseteq D$ be a $(v_1,v_2)$-path and $P_2\subseteq D$ be a $(v_2,v_1)$-path which are edge-disjoint.
Let $k\geq\max_{i\in[2]}|V(P_i)\cap \dot{A}|$.
Let $\mathcal{A}=(E^{\mathrm{ab}},f)$ be a $(\{v_1,v_2\},k+1)$-absorbing structure such that, for each $i\in[2]$, $E(P_i)\cap E^{\mathrm{ab}} = \varnothing$.
Then, for each $i\in[2]$ there exists a set $E_i\subseteq f(v_i)$, where $|E_i|=1$ if $v_i\in \dot{A}$ and $|E_i|=2$ otherwise, such that the digraph with edge set $E(P_1)\cup E(P_2)\cup E_1 \cup E_2$ can be decomposed into two $(A^+,A^-)$-paths.
\end{lemma}

\begin{proof}
For each $i\in[2]$, we consider three cases.
If $v_i\in A^+$, by our choice of $k$, there is some edge $v_iy_i\in f(v_i)$ with $y_i\notin V(P_{3-i})$\COMMENT{By pigeonhole, we have $y_i\notin V(P_{3-i})\cap\dot{A}$. Here we are implicitly using the definition of $\mathcal{A}$ in the sense that $y_i\in\dot{A}$ to derive the conclusion.}.
In such a case, we let $P_i^+\coloneqq v_iy_i$ and $P_i^-\coloneqq v_i$.
If $v_i\in A^-$, similarly, there is some edge $x_iv_i\in f(v_i)$ with $x_i\notin V(P_i)$, and we let $P_i^+\coloneqq v_i$ and $P_i^-\coloneqq x_iv_i$.
Otherwise, we have $v_i\in A^0$ and, again by assumption, there must be two edges $x_iv_i,v_iy_i\in f(v_i)$ such that $x_i\notin V(P_i)$ and $y_i\notin V(P_{3-i})$.
In this case, we let $P_i^+\coloneqq v_iy_i$ and $P_i^-\coloneqq x_iv_i$.
In all cases we set $E_i\coloneqq E(P_i^+)\cup E(P_i^-)$.

Now let $P\coloneqq P_1^-P_1P_2^+$ and $P'\coloneqq P_2^-P_2P_1^+$.
Clearly, $P$ and $P'$ decompose $E(P_1)\cup E(P_2)\cup E_1\cup E_2$.
Furthermore, both $P$ and $P'$ are $(A^+,A^-)$-paths by the definition of $\mathcal{A}$ and our choice of $x_i,y_i$.
Indeed, for each $i\in[2]$, by definition we have that the first vertex of $P_i^-$ lies in $A^+$, and the last vertex of $P_i^+$ lies in $A^-$, which immediately yields the result.
\end{proof}

The following lemma shows how an absorbing structure can be used to absorb a collection of short cycles.
\begin{lemma}
\label{lem:absorbc3}
Let $D$ be an $(n,p,\kappa,\lambda)$-digraph with $\kappa \geq 4N/n$.
Let $\mathcal{C}_3$ be a collection of at most $N$ edge-disjoint cycles such that, for each $C\in \mathcal{C}_3$, we have  $|V(C)\cap\dot{A}|\leq \kappa$.
Let $\mathcal{A}=(E^{\mathrm{ab}},f)$ be an $(A^0\cup \dot{A},3\kappa)$-absorbing structure with $E(\mathcal{C}_3)\cap E^{\mathrm{ab}}=\varnothing$.
Then, the digraph with edge set $E(\mathcal{C}_3)\cup E^{\mathrm{ab}}$ can be decomposed into a set of cycles $\mathcal{C}^*$ and a digraph $Q$ such that
\begin{enumerate}[label=$(\mathrm{S}\arabic*)$]
\item\label{short1} $E(\mathcal{C}^*) \subseteq E(\mathcal{C}_3)$;
\item\label{short2} for all $C\in \mathcal{C}^*$ we have $|V(C)\cap \dot{A}|>\kappa$, and 
\item\label{short3} $Q$ has a perfect decomposition in which each path is an $(A^+,A^-)$-path.
\end{enumerate}
\end{lemma}

\begin{proof}
We will construct $Q$ and $\mathcal{C}^*$ over multiple rounds. 
We start with a set of cycles $\mathcal{C}\coloneqq\mathcal{C}_3$ and a set of edges $F^{\mathrm{ab}}\coloneqq E^{\mathrm{ab}}$, and we set $Q\coloneqq (V(D), \varnothing )$ and $\mathcal{C}^*\coloneqq \varnothing$.
In each round, we will update $\mathcal{C}$, $F^{\mathrm{ab}}$, $Q$ and $\mathcal{C}^*$ by moving some edges from $E(\mathcal{C})\cup F^{\mathrm{ab}}$ to $E(Q)\cup E(\mathcal{C}^*)$.
In particular, in each round, we will combine edges from $F^{\mathrm{ab}}$ with some Eulerian subdigraph of $E(\mathcal{C})$ to form $(A^+,A^-)$-paths (by using   \cref{lem:findleavingedges}) and move the (edges of these) paths into $Q$.
Since we only ever add $(A^+,A^-)$-paths to $Q$, then $Q$ always has a perfect decomposition by \cref{pr:PartDecomp2}.
(Throughout we will also maintain that $F^{\mathrm{ab}}$ can be decomposed into $(A^+, A^-)$-paths.)
After these paths have been added to $Q$, what remains of $E(\mathcal{C})$ will be Eulerian and reside on a significantly smaller number of vertices.
We will then apply \cref{th:compress} to decompose what remains of $E(\mathcal{C})$ into cycles: any medium or long cycle in this decomposition (i.e., those that have more than $\kappa$ vertices in $\dot{A}$) will be added to $\mathcal{C}^*$, while the remaining cycles in the decomposition form the set $\mathcal{C}$ for the next round. 
Since $|V(\mathcal{C})|$ decreases in each round, this process will stop after a finite number of rounds.
At that point, we add any remaining edges from $F^{\mathrm{ab}}$, decomposed into $(A^+,A^-)$-paths, into $Q$, which will have a perfect decomposition.

It is important that we use edges/paths from our absorbing structure carefully in each round so that there are sufficiently many choices available at each vertex in future rounds.
By solving a suitable flow problem, we will make sure that, over the course of all rounds, we use at most $\kappa$ edges/paths from $E^{\mathrm{ab}}$ at each vertex.
This will ensure there are always at least $2 \kappa$ choices of edges/paths available in $F^{\mathrm{ab}}$ at every vertex in every round, which will allow us to construct suitable absorbing (sub)structures in order to apply \cref{lem:findleavingedges}.

Let us now give the details of this iterative process.
At the start of each round we are given a digraph $Q$, a set of edges $F^{\mathrm{ab}} \subseteq E^{\mathrm{ab}}$ and two sets of cycles $\mathcal{C}$ and $\mathcal{C}^*$, which have been updated in previous rounds and satisfy the following properties:
\begin{enumerate}[label=(\alph*)]
    \item\label{item:disjoint} 
     $E(\mathcal{C}_3)\cup E^{\mathrm{ab}}$ is the disjoint union of $E(Q)$, $F^{\mathrm{ab}}$, $E(\mathcal{C})$, and $E(\mathcal{C}^*)$; 
    \item\label{item:Q} $Q$ can be decomposed into $(A^+, A^-)$-paths;
    \item\label{item:propa} writing $n' \coloneqq |V(\mathcal{C})|$, we have $|\mathcal{C}|\leq c'n'\log n'$ (where $c'$ is the constant from \cref{th:compress}) and $|V(C) \cap \dot{A}| \leq \kappa$ for all $C \in \mathcal{C}$, and
    \item\label{item:Cstar} $|V(C) \cap \dot{A}|>\kappa$ for all $C \in \mathcal{C}^*$. 
\end{enumerate}
The digraph $Q$ and the sets $F^{\mathrm{ab}}$ and $\mathcal{C}$ are updated several times throughout each round, and the notation will always refer to their updated form. 

Recall that, as stated in \cref{def:absstruct}, we may think of $\mathcal{A} = (E^{\mathrm{ab}}, f)$ as a set of edge-disjoint paths of length $1$ or $2$.
In the same way, we also think of the edges of $F^{\mathrm{ab}}$ as paths of length $1$ or $2$.
For any $v \in A^+ \cup A^-$, we think of each edge in $F^{\mathrm{ab}} \cap f(v)$ as an $(A^+, A^-)$-path of length $1$.
Because of the way we use edges for absorption (i.e., by using \cref{lem:findleavingedges}), for any $v \in A^0$, the set $F^{\mathrm{ab}} \cap f(v)$ will always contain the same number of edges from $A^+$ to $v$ as from $v$ to $A^-$, and these will be (implicitly) paired up arbitrarily and thought of as $(A^+, A^-)$-paths of length $2$.
Note that the pairing is updated (arbitrarily) every time $F^{\mathrm{ab}}$ is updated.
For each vertex $v$, let $a(v)$ denote the current number of available paths in $F^{\mathrm{ab}} \cap  f(v)$, that is, 
\[a(v)=\begin{cases}d^+_{F^{\mathrm{ab}}\cap f(v)}(v)&\text{if }v\in A^+,\\
d^-_{F^{\mathrm{ab}}\cap f(v)}(v)&\text{if }v\in A^-,\\
d^+_{F^{\mathrm{ab}}\cap f(v)}(v) = d^-_{F^{\mathrm{ab}}\cap f(v)}(v)&\text{if }v\in A^0.\end{cases}\]
As we want to use at most $\kappa$ paths at each vertex $v\in V(\mathcal{C})$, we define the number of \emph{ready paths} at $v$ as $r(v)\coloneqq a(v)-2\kappa$.
Throughout, we implicitly update the values of $a(v)$ and $r(v)$ each time we update $F^{\mathrm{ab}}$.

We further assume the following property about $\mathcal{C}$ at the start of the round:
\begin{enumerate}[label=(\alph*),start=5]
    \item\label{item:propb} for all $v\in V(\mathcal{C})$ we have at least one of $d^+_{\mathcal{C}}(v)\leq r(v)$, or $r(v)=\kappa$ (i.e., the number of cycles in $\mathcal{C}$ passing through $v$ is bounded above by $r(v)$ or $r(v) = \kappa$).
\end{enumerate}

Note that, at the start of the first round, we have $Q=(V(D), \varnothing)$, $F^{\mathrm{ab}}=E^{\mathrm{ab}}$, $\mathcal{C} = \mathcal{C}_3$ and $\mathcal{C}^* = \varnothing$, so \ref{item:disjoint}--\ref{item:propb} hold.

We now show how to update $Q$, $\mathcal{C}$, and $\mathcal{C}^*$ and check that \ref{item:disjoint}--\ref{item:propb} hold at the end of the round.
Consider the flow problem $\mathit{FP}(\mathcal{C}; 2,\kappa)$ and let $\phi$ be a maximum integer flow.
Let $F_\phi$ be the residual digraph of $F=F(\mathcal{C})$ under~$\phi$.
Set
$T\coloneqq\{v\in V(\mathcal{C}): F_\phi \text{ contains an }(s,v)\text{-path}\}$ and $T'\coloneqq V(\mathcal{C})\setminus T$.

We establish a bound on $|T|$ for later.
Since the cut-set $M_0\coloneqq \{sC:C\in \mathcal{C}\}$, by \ref{item:propa}, has capacity $2|\mathcal{C}|\leq 2c'n'\log n'$, the max-flow min-cut theorem (\cref{thm:maxflow}) implies that $\mathit{val}(\phi) \leq 2c'n'\log n'$.
Furthermore, all vertices in $T$ must have $\kappa$ units of flow going through them in $\phi$, as otherwise we would immediately be able to increase the flow.
Therefore,
\begin{equation*}
    \kappa|T|\leq\mathit{val}(\phi)\leq 2c'n'\log n',
\end{equation*}
which implies
\begin{equation}
    \label{eq:lemc5:1}
    |T|\leq \frac{2c'n'\log n'}{\kappa}\leq \frac{n'}{2},
\end{equation}
as $\kappa \geq 4N/n\geq 4c'\log n'$.

We use $\phi$ to assign vertices to cycles as follows.
First, we greedily decompose $\phi$ into single-unit flows.
As each single-unit flow goes through one cycle $C\in \mathcal{C}$ and one vertex $v\in V(C)$, we understand this as assigning $v$ to $C$.
Note that for every $v \in V(\mathcal{C})$, the flow $\phi(vt)$ through the edge $vt$ satisfies
\begin{equation}
\label{eq:flow}
\phi(vt) \leq \min\{d^+_{\mathcal{C}}(v), \kappa\} \leq r(v), 
\end{equation}
where the last inequality holds by \ref{item:propb}.

We partition $\mathcal{C}$ into three sets $\mathcal{C} = \mathcal{C}^0 \cup \mathcal{C}^1 \cup \mathcal{C}^2$, where $\mathcal{C}^i$ is the set of cycles $C\in\mathcal{C}$ that are assigned exactly $i$ vertices from $T'$.
Recall that we decomposed the flow $\phi$ into single-unit flows.
For each $i\in[2]_0$, let $\phi^i$ be the flow that is given by the sum of the single-unit flows of the decomposition that pass through cycles in $\mathcal{C}^i$.
In particular, this means that
$\phi = \phi^1 + \phi^2 + \phi^3$ and, for each $v\in V(\mathcal{C})$, $\phi^i(vt)$ is the number of cycles in $\mathcal{C}^i$ to which $v$ is assigned.
We next show how to process the cycles in each $\mathcal{C}^i$, but first we need the following claim.

\begin{claim}
\label{cl:maxflow}
For all cycles $C\in \mathcal{C}^1$, we have $|V(C) \cap T'| = 1$ (and so the unique vertex in $V(C) \cap T'$ must be assigned to $C$).
In particular, for all $v \in T'$ we have $\phi^1(vt)  = d^+_{\mathcal{C}^1}(v)$.

For all cycles $C\in \mathcal{C}^0$, we have $|V(C) \cap T'| = 0$.
\end{claim}

\begin{claimproof}
For all $C\in \mathcal{C}^0 \cup \mathcal{C}^1$, note first that there is a path from $s$ to $C$ in $F_\phi$.
Indeed, if $C$ is assigned fewer than two vertices, then the path is immediate, while if $C$ is assigned two vertices, at least one of them, say $u$, is in $T$, and so the $(s,u)$-path in $F_\phi$ (which exists by the definition of $T$) can be extended to $C$.
Now any vertices $v\in V(C)$ that are not assigned to $C$ must lie in $T$ by definition, as we can extend the $(s,C)$-path in $F_\phi$ to $v$.
Therefore, for each $i\in\{0,1\}$ and all $C \in \mathcal{C}^i(v)$ we must have $|V(C) \cap T'| = i$.

Now, any vertex $v \in T'$ that belongs to a cycle $C \in \mathcal{C}^1$ is also assigned to it, establishing that $\phi^1(vt)  = d^+_{\mathcal{C}^1}(v)$ for all $v \in T'$. 
\end{claimproof}

We start by processing the cycles in $\mathcal{C}^2$.
For each cycle $C\in \mathcal{C}^2$, let $v_1,v_2 \in T'$ be such that $\phi(Cv_i)=1$ for each $i\in[2]$, i.e., these are the vertices assigned to $C$.
We split $C$ into a $(v_1,v_2)$-path $P_{12}$ and a $(v_2,v_1)$-path $P_{21}$.
We select any $\kappa +1$ available paths at each $v_i$ from $F^{\mathrm{ab}}$ to define a $(\{v_1,v_2\},\kappa+1)$-absorbing structure $\mathcal{A}_C$ (we show below that this is always possible). 
We then apply \cref{lem:findleavingedges} to the paths $P_{12},P_{21}$ and the absorbing structure $\mathcal{A}_C$ with $k= \kappa$.
Thus, for each $i\in[2]$ we obtain an available path $E_i \subseteq f(v_i) \cap F^{\mathrm{ab}}$ such that $E(P_{12}) \cup E(P_{21}) \cup E_1 \cup E_2$ can be decomposed into two $(A^+, A^-)$-paths $P_1'$ and $P_2'$.
For each $i\in[2]$, we add the edges of $P_i'$ to $Q$, remove $E_i$ from $F^{\mathrm{ab}}$ and remove $C$ from $\mathcal{C}$.
We repeat this for all cycles in $\mathcal{C}^2$. 

We now check that it is always possible to find the desired absorbing structure $\mathcal{A}_C$.
Notice that, in order to process $\mathcal{C}^2$, the number of available paths that we use at any vertex $v$ is the number of cycles of $\mathcal{C}^2$ to which $v$ is assigned, which at the start of the round is $\phi^2(vt)\leq\min\{d^+_\mathcal{C}(v), \kappa\}
\leq r(v) = a(v) - 2\kappa$ (by \eqref{eq:flow}).
This means there are always $2 \kappa$ available paths at every vertex each time we apply \cref{lem:findleavingedges}.

After processing $\mathcal{C}^2$, for any vertex $v\in T'$, we have used at most $\phi^2(vt)$ available paths from~$f(v)$.
Recalling that we always update $a(v)$, we now have for any $v \in T'$ that
\begin{equation}
\label{eq:avedges}
a(v)\geq \phi(vt) + 2\kappa - \phi^2(vt) = \phi^1(vt) + 2\kappa = d^+_{\mathcal{C}^1}(v) + 2\kappa,
\end{equation}
where we have used \eqref{eq:flow} for the first inequality and \cref{cl:maxflow} for the last equality.
The first equality holds as $\phi^0(vt)=0$ by definition, since $v\in T'$.
Note that \ref{item:propb} holds for the current value of $a(v)$ and the current set of cycles $\mathcal{C}=\mathcal{C}^0 \cup \mathcal{C}^1$, since $a(v)$ is unchanged for $v \in T$, and that \eqref{eq:avedges} confirms \ref{item:propb} for $v \in T'$ (by using \cref{cl:maxflow} to note that $d^+_{\mathcal{C}^0}(v) = 0$ for all $v \in T'$). 

Next we process cycles in $\mathcal{C}^1$.
Recall that, by \cref{cl:maxflow}, such cycles contain exactly one vertex of $T'$.
Let $R$ be an empty set of edges; this set will be updated while processing $\mathcal{C}^1$ and will always form an Eulerian digraph.
We say a pair of cycles $C_1, C_2 \in \mathcal{C}^1$ is $T$-intersecting if
$\varnothing \not= V(C_1) \cap V(C_2) \subseteq T$ (and thus their unique vertices in $T'$ are distinct).
Whenever we have a $T$-intersecting pair of cycles $C_1, C_2 \in \mathcal{C}^1$, we process them as follows.
Let $v_1\neq v_2$ be the vertices of $C_1$ and $C_2$ in $T'$, respectively.
Starting from $v_1$, let $v'_1$ be the first vertex along $C_1$ in $V(C_1)\cap V(C_2)$ and define $P_{12}\coloneqq v_1C_1v'_1C_2v_2$.
Define $v'_2$ analogously, and let $P_{21}\coloneqq v_2C_2v'_2C_1v_1$.
It is easy to see that $P_{12}$ and $P_{21}$ are edge-disjoint.
Again, we construct a $(\{v_1,v_2\},2\kappa+1)$-absorbing structure $\mathcal{A}_{C_1C_2}$  by taking $2\kappa+1$ available paths at each $v_i$ from $F^{\mathrm{ab}}$; this is always possible by \eqref{eq:avedges}, as we find an absorbing structure for $v_i$ at most $d^+_{\mathcal{C}^1}(v_i)$ times.
We apply \cref{lem:findleavingedges} to the paths $P_{12},P_{21}$ and the absorbing structure $\mathcal{A}_{C_1C_2}$ with $k= 2\kappa$ to obtain available paths $E_i \subseteq f(v_i) \cap F^{\mathrm{ab}}$, for $i\in[2]$, such that $E(P_{12}) \cup E(P_{21}) \cup E_1 \cup E_2$ can be decomposed into two $(A^+,A^-)$-paths $P_1'$ and $P_2'$.
For each $i\in[2]$, we add the edges of $P_i'$ to $Q$ and remove the edges of $E_i$ from $F^{\mathrm{ab}}$. 
The remaining edges of the cycles $C_1$ and $C_2$, namely $(E(C_1) \cup E(C_2)) \setminus (E(P_{12}) \cup E(P_{21}))$, are then added to the residual digraph $R$.
Notice that the set of edges added to $R$ is Eulerian, so $R$ remains Eulerian.
Furthermore, note that all edges added to $R$ have both endpoints in $T$.
Finally, we remove $C_1$ and $C_2$ from $\mathcal{C}$ (and from $\mathcal{C}^1$).

We repeat this as long as we can find a $T$-intersecting pair of cycles in $\mathcal{C}^1$.
When no such pair can be found, then, among the remaining cycles of $\mathcal{C}^1$, any two either share a vertex in $T'$ or are vertex-disjoint.
This implies that, at this stage, the set $\overline{T}\coloneqq V(\mathcal{C}^1) \cap T'$ satisfies $|\overline{T}| \leq |T|/2$.
(To see this, for each vertex  $v \in \overline{T}$, pick a cycle $C_v \in \mathcal{C}^1$ containing $v$.
Notice that each such cycle has all its (at least two) remaining vertices in $T$ and, furthermore, the cycles $C_v$ are vertex-disjoint.)
We move all the remaining cycles of $\mathcal{C}$ (i.e., all that remain in $\mathcal{C}^1$ and all in $\mathcal{C}^0$) to $R$.
Then, $R$ is Eulerian and $V(R) \subseteq T \cup \overline{T}$ (recall any cycle in $\mathcal{C}^0$ has all its vertices in $T$ by \cref{cl:maxflow}).
Then,
\begin{equation}
\label{eq:R}
n''\coloneqq |V(R)| \leq |T| + |\overline{T}| \leq 3|T|/2 \leq 3n'/4,
\end{equation}
where the last inequality follows by \eqref{eq:lemc5:1}.

Now we decompose $R$ into at most $c'n'' \log n''$ cycles using \cref{th:compress}; any resulting cycles with more than $\kappa$ vertices in $\dot{A}$ are added to $\mathcal{C}^*$, while all other cycles are added to (the currently empty) $\mathcal{C}$.
This completes the round and the description of the sets $\mathcal{C}$, $\mathcal{C}^*$, $F^{\mathrm{ab}}$, and $Q$ ready for the next round.
Notice that at the end of the round $V(\mathcal{C})$ is smaller than at the start, by \eqref{eq:R}.
It remains to check that \ref{item:disjoint}--\ref{item:propb} hold.

It immediately follows by construction that \ref{item:disjoint}--\ref{item:Cstar} hold (\ref{item:disjoint} holds because we only  move edges between the sets, and \ref{item:Q} holds because we only add $(A^+,A^-)$-paths to $Q$).
Finally, we prove that \ref{item:propb} holds too. As noted after \eqref{eq:avedges}, we know \ref{item:propb} holds after $\mathcal{C}^2$ is processed. After that, when processing $\mathcal{C}^1$, whenever an application of Lemma~\ref{lem:findleavingedges}  reduces $a(v)$ by $1$, it also reduces $d^+_{\mathcal{C}}(v)$ by $1$,  so condition \ref{item:propb} is maintained to the end of the round.

Thus, we may iterate the described process through the rounds, until we obtain the final sets $Q$, $\mathcal{C} = \varnothing$, $\mathcal{C}^*$ and $F^{{\rm ab}}$ satisfying \ref{item:disjoint}--\ref{item:propb}.
(Recall that the process must terminate since, by \eqref{eq:R}, the set of cycles that is considered for each subsequent round is contained in a smaller set of vertices than the previous.)
The remaining paths of $F^{{\rm ab}}$ are $(A^+,A^-)$-paths; these paths are removed from $F^{{\rm ab}}$ and added to $Q$. 

It is straightforward to check that $Q$ and $\mathcal{C}^*$ now satisfy the conclusion of the lemma.
Indeed, over the course of all rounds, we moved all edges from $E(\mathcal{C}_3) \cup E^{\mathrm{ab}}$ to $E(Q) \cup E(\mathcal{C}^*)$.
At every stage, $Q$ was updated by adding $(A^+,A^-)$-paths (which gives a perfect decomposition of $Q$ by \cref{pr:PartDecomp2}), and $\mathcal{C}^*$ was updated by adding cycles that have more than $\kappa$ vertices in $\dot{A}$.
\end{proof}

We are finally ready to prove the main result.

\begin{proof}[Proof of \cref{thm:main}]
Recall that $D$ is either an $(n,p,\kappa, \lambda)$-digraph satisfying \ref{thmitem1}--\ref{thmitem3} or an $(n,p,\kappa, \lambda)$-pseudorandom digraph satisfying \ref{thmitem1'}, \ref{thmitem2'}, and \ref{thmitem3}, with $n \geq n_0$ (for a suitably large choice of $n_0$).
We work with both cases simultaneously.

First, one can easily check that the conditions \ref{thmitem1}--\ref{thmitem3} together with $n \geq n_0$, for a sufficiently large $n_0$, imply the conditions \ref{itemthmproofa}--\ref{itemthmproofc} below, which are precisely the parameter conditions required in order to apply
\cref{lem:absstruct1,lem:absstruct2,lem:absorbc1,lem:absorbc2,lem:absorbc3} to an $(n,p,\kappa, \lambda)$-digraph:
\begin{enumerate}[label=$(\mathrm{\alph*})$]
    \item\label{itemthmproofa} $\max\{ 100\log n, 12,(12\lambda)^{1/2}, (72 N^2)^{1/5},{4N}/{n}\}\leq \kappa < \min\{{np}/{120},N^{1/2}\} $,
    \item\label{itemthmproofb} $\lambda \leq {np}/{3}$,
    \item\label{itemthmproofc} $4np\log(2n)\leq \kappa\lambda$,
\end{enumerate}
where $N\coloneqq c'n \log n$ and $c'$ is the constant from \cref{th:compress}.
\COMMENT{
Recall that the conditions \ref{thmitem1}--\ref{thmitem3} are 
\ref{thmitem1} $\kappa=3N^{2/5}$, 
\ref{thmitem2} $np\geq365N^{2/5}$, 
\ref{thmitem3} $\lambda=\min\{np/3,\kappa^2/12\}$.
\\
All the following hold for $n \geq n_0$ for $n_0$ suitably large.
For \ref{itemthmproofa}, note that the maximum on the left hand side is dominated by $\max( (12\lambda)^{1/2}, (72N^2)^{1/5})$, which is at most $\kappa$ by \ref{thmitem1} and \ref{thmitem3}. The upper bound in \ref{itemthmproofa} holds by noting $\kappa \leq np/120$ by \ref{thmitem1} and \ref{thmitem2}, and $\kappa < N^{1/2}$ by \ref{thmitem1}.
\ref{itemthmproofb} holds by \ref{thmitem3}.
For \ref{itemthmproofc}, if $\lambda = np/3$, then $\kappa \lambda \geq 4np \log(2n)$ by \ref{thmitem1}, and if $\lambda = \kappa^2/12$, then $\kappa \lambda = \kappa^3/12 \geq N^{6/5} \geq 4np \log(2n)$.}
Similarly, one can easily check that the conditions \ref{thmitem1'}, \ref{thmitem2'}, and \ref{thmitem3} together with $n \geq n_0$, for a sufficiently large~$n_0$, imply the  conditions \ref{itemthmproofa'}, \ref{itemthmproofb}, and \ref{itemthmproofc} (with \ref{itemthmproofa'} given below), which are precisely the parameter conditions required in order to apply
\cref{lem:absstruct1,lem:absstruct2,lem:absorbc1,lem:absorbc2,lem:absorbc3} to an $(n,p,\kappa, \lambda)$-pseudorandom digraph:
\begin{enumerate}[label=$(\mathrm{\alph*}')$]
    \item\label{itemthmproofa'} $\max\{ 100\log n, 12,(12\lambda)^{1/2}, (7200 N^2p)^{1/5},\frac{4N}{n}, \sqrt{12/(25p)}\log n\}\leq \kappa < \min\{\frac{np}{120},N^{1/2}\} $.
\end{enumerate}
The pseudorandom case only makes a difference for \cref{lem:absorbc2}.
\COMMENT{
Recall that the conditions \ref{thmitem1'}, \ref{thmitem2'}, and \ref{thmitem3} are 
\ref{thmitem1'} $\kappa = 6(N^2 p)^{1/5}$, 
\ref{thmitem2'} $p \geq n^{-1/3} \log^4 n$, and
\ref{thmitem3} $\lambda = \min\{np/3, \kappa^2 / 12\}$.
\\
All the following hold for $n \geq n_0$ for $n_0$ suitably large.
For \ref{itemthmproofa'}, the maximum on the left hand side is dominated by $\max ((12\lambda)^{1/2}, (7200N^2p)^{1/5})$ (where we can exclude $\sqrt{12/(25p)} \log n$ by \ref{thmitem2'}). We see $\kappa$ is bigger than this by \ref{thmitem3} and \ref{thmitem1'}. For the upper bound we have $\kappa \leq N^{1/2}$ by \ref{thmitem1'} and $\kappa \leq np/120$ by \ref{thmitem1'} and \ref{thmitem2'}.
\ref{itemthmproofb} holds by \ref{thmitem3} again.
For \ref{itemthmproofc}, if $\lambda = np/3$, then $\kappa \lambda \geq 4np\log(2n)$ (using \ref{thmitem1'} and \ref{thmitem2'}), and if $\lambda = \kappa^2 /12$ then $\kappa \lambda = \kappa^3/12 \geq N^{6/5}p^{3/5}$ by \ref{thmitem1'}, which is at least  $4np\log(2n)$.
}

For the $(n,p,\kappa, \lambda)$-(pseudorandom) digraph $D$, let $A^+ \cup A^- \cup A^0$ be the associated partition of $V(D)$.
Write $B^+$, $B^-$, and $B^0$ for the set of vertices $v \in V(D)$ such that $\ex_D(v)>0$, $\ex_D(v)<0$, and $\ex_D(v)=0$, respectively.
From \cref{def:digraphclass1}, clearly $A^+ \subseteq B^+$ and $A^- \subseteq B^-$. 

Let $\dot{\mathcal{A}}=(\dot{E}^{\mathrm{ab}},\dot{f})$ be an $(\dot{A},12\kappa )$-absorbing structure contained in $D$, which exists by \cref{lem:absstruct1}, and let $\mathcal{A}^0=(E_0^{\mathrm{ab}},f_0)$ be an $(A^0,3\kappa )$-absorbing structure contained in $D$, which exists by \cref{lem:absstruct2}.
Note that these two absorbing structures must be edge-disjoint by definition.
We next split up $\dot{\mathcal{A}}$ into an $(\dot{A},7\kappa -1)$-absorbing structure $\dot{\mathcal{A}}_1=(\dot{E}^{\mathrm{ab}}_1,\dot{f}_1)$, an $(\dot{A},2\kappa +1)$-absorbing structure $\dot{\mathcal{A}}_2=(\dot{E}^{\mathrm{ab}}_2,\dot{f}_2)$, and an $(\dot{A},3\kappa )$-absorbing structure $\dot{\mathcal{A}}_3=(\dot{E}^{\mathrm{ab}}_3,\dot{f}_3)$.
To do so, for each $v\in \dot{A}$, we arbitrarily split the $12\kappa$ edges in $\dot{f}(v)$ into sets of size $7\kappa -1$, $2\kappa +1$ and $3\kappa$ and set these to be $\dot{f}_1(v)$, $\dot{f}_2(v)$ and $\dot{f}_3(v)$, respectively, and set $\dot{E}^{\mathrm{ab}}_i\coloneqq \bigcup_{v\in \dot{A}}\dot{f}_i(v)$ for each $i\in[3]$.
Lastly, we combine $\dot{\mathcal{A}}_3$ and $\mathcal{A}^0$ into an $(\dot{A}\cup A^0, 3\kappa)$-absorbing structure $\mathcal{A}_3=(\dot{E}^{\mathrm{ab}}_3 \cup E_0^{\mathrm{ab}}, f_3)$, where $f_3\rvert_{\dot{A}}=\dot{f}_3$ and $f_3\rvert_{A^0}=f_0$.

Consider a set of paths which consists of every individual edge in $\dot{E}^{\mathrm{ab}}$ and a partition of the edges in $E^{\mathrm{ab}}_0$ into paths of length two.
Each path is an $(A^+,A^-)$-path and, therefore, a $(B^+,B^-)$-path.
Moreover, note that, by \cref{lem:absstruct1,lem:absstruct2}, $\dot{E}^{\mathrm{ab}} \cup E^{\mathrm{ab}}_0$ contains at most $150\kappa + 5\kappa = 155\kappa$ edges incident to each $v\in \dot{A}$.
This means (by \ref{def:digraphitem1} and \ref{def:digraphitem2}) that removing all these paths from $D$ will not change the sign of the excess of any vertex $v \in V(D)$, that is, if we write $D'\coloneqq D\setminus(\dot{E}^{\mathrm{ab}} \cup E^{\mathrm{ab}}_0)$, then a vertex of positive (resp.\ negative) excess in $D'$ belongs to $B^+$ (resp.\ $B^-$).

Next, we greedily remove paths from $D'$ that start in vertices with positive excess in $D'$ and end in vertices with negative excess in $D'$ until this is no longer possible. We call the set of these paths $\mathcal{P}'$ (so every path in $\mathcal{P}'$ is a $(B^+,B^-)$-path) and set $D^*\coloneqq D'\setminus E(\mathcal{P}')$ (so that $\ex(D^*)=0$ by \cref{pr:PartDecomp1}). 
 
 We apply \cref{th:compress} to every component of $D^*$ and obtain a decomposition $\mathcal{C}$ of the edges of $D^*$ into at most $N\coloneqq c'n\log n$ cycles. Let
\begin{align*}
    \mathcal{C}_1&\coloneqq \{C\in \mathcal{C}: |V(C)\cap \dot{A}|\geq N/\kappa\},\\
    \mathcal{C}_2&\coloneqq \{C\in \mathcal{C}: \kappa < |V(C)\cap \dot{A}| < N/\kappa \},\text{ and} \\
    \mathcal{C}_3&\coloneqq \{C\in \mathcal{C}:  |V(C)\cap \dot{A}| \leq \kappa \}. 
\end{align*}
At this point, we have
\begin{align*}
    E(D) &= E(D') \cup \dot{E}^{\rm ab} \cup E_0^{\rm ab} \\
    &= E(D') \cup \dot{E}_1^{\rm ab}  \cup \dot{E}_2^{\rm ab} \cup \dot{E}_3^{\rm ab} \cup E_0^{\rm ab}  \\
    &= E(\mathcal{P}') \cup E(D^*) \cup \dot{E}_1^{\rm ab}  \cup \dot{E}_2^{\rm ab} \cup \dot{E}_3^{\rm ab} \cup E_0^{\rm ab} \\
    &= E(\mathcal{P}') \cup \left( E(\mathcal{C}_1) \cup \dot{E}_1^{\rm ab} \right)  \cup \left( E(\mathcal{C}_2) \cup \dot{E}_2^{\rm ab} \right) \cup \left( E(\mathcal{C}_3) \cup \dot{E}_3^{\rm ab} \cup E_0^{\rm ab} \right).
\end{align*}

Noting that $|\mathcal{C}_3|\leq N$, we apply \cref{lem:absorbc3} to $\mathcal{C}_3$ and $\mathcal{A}_3$ to decompose the edges of $E(\mathcal{C}_3) \cup \dot{E}_3^{\rm ab} \cup E_0^{\rm ab}$ into a set of cycles $\mathcal{C}^*_3$ and a digraph $Q$, where $Q$ has a perfect decomposition $\mathcal{P}_3$ into $(A^+,A^-)$-paths, $|\mathcal{C}^*_3|\leq |\mathcal{C}_3|$, and for all $C\in \mathcal{C}^*_3$ we have $|V(C)\cap \dot{A}|>\kappa$.
(Indeed, the fact that $|\mathcal{C}_3^*| \leq |\mathcal{C}_3|$ follows from conclusions \ref{short1} and \ref{short2} of \cref{lem:absorbc3}.
To see this, note that any cycle $C\subseteq D$ satisfies $|\{uv\in E(C):v\in\dot{A}\}|=|V(C) \cap \dot{A}|$.
Thus, by \ref{short2}, for each $C\in\mathcal{C}_3^*$ and each $C'\in\mathcal{C}_3$ we must have $|\{uv\in E(C):v\in\dot{A}\}|>|\{uv\in E(C'):v\in\dot{A}\}|$, so by \ref{short1}, $\mathcal{C}^*$ must have fewer cycles than $\mathcal{C}_3$.)
Let $\mathcal{C}^*_1\coloneqq \{C\in \mathcal{C}^*_3: |V(C)\cap \dot{A}|\geq N/\kappa\}$ and $\mathcal{C}^*_2\coloneqq \{C\in \mathcal{C}: \kappa  < |V(C)\cap \dot{A}| < N/\kappa \}$ and note that, as $|\mathcal{C}^*_3|\leq |\mathcal{C}_3|$, we have $|\mathcal{C}^*_1|,|\mathcal{C}^*_2|\leq N$.

Next, we apply \cref{lem:absorbc1} to $\mathcal{C}_1\cup \mathcal{C}^*_1$ and $\dot{\mathcal{A}}_1$; this shows that the digraph with edge set $E(\mathcal{C}_1\cup \mathcal{C}^*_1)\cup \dot{E}^{\mathrm{ab}}_1$ has a perfect decomposition $\mathcal{P}_1$ into $(A^+,A^-)$-paths.

In the same way, applying \cref{lem:absorbc2} to $\mathcal{C}_2\cup \mathcal{C}^*_2$ and $\dot{\mathcal{A}}_2$ shows that the digraph with edge set $E(\mathcal{C}_2\cup \mathcal{C}^*_2)\cup \dot{E}^{\mathrm{ab}}_2$ has a perfect decomposition $\mathcal{P}_2$ into $(A^+,A^-)$-paths.

Now it is easy to check that $\mathcal{P}' \cup \mathcal{P}_1 \cup \mathcal{P}_2 \cup \mathcal{P}_3$ is a decomposition of $E(D)$ into paths, and every path is a $(B^+, B^-)$-path, so this is a perfect decomposition of $D$ by \cref{pr:PartDecomp2}.
\end{proof}

%%%%%%%%%%%%%%%%%%%%%%%%%%%%%%%%%%%%%%%%%%%%%%%%%%%%%%%%%%%%%%%%%%%%%
%%%%%%%%%%%%%%%%%%%%%%%%%%%%%%%%%%%%%%%%%%%%%%%%%%%%%%%%%%%%%%%%%%%%%

\section{Path decompositions of random digraphs}
\label{sec:decomprandom}

In this section we derive \cref{thm:mainintro}.
This will follow immediately as a consequence of \cref{thm:main} and the following result.

\begin{theorem}\label{thm:random}
Let $13\log^2n/\sqrt{n}\leq p\leq1-150\log^4n/n$.
Let $\kappa\coloneqq\sqrt{np(1-p)}/(155\log^{3/4}n)$ and $\lambda\coloneqq5\sqrt{n/(1-p)}\log^2n$.
Then, a.a.s.~$D_{n,p}$ is an $(n,p,\kappa,\lambda)$-pseudorandom digraph.
\end{theorem}

Now we prove \cref{thm:mainintro}.

\begin{proof}[Proof of \cref{thm:mainintro}]
Let $\log^4n/n^{1/3}\leq p \leq1-\log^{5/2}n/n^{1/5}$ (so within the range stated in the theorem), and let $n$ be sufficiently large.
As usual, let $N \coloneqq c'n \log n$, where $c'$ is the constant from \cref{th:compress}.

If we let $D=D_{n,p}$, then by \cref{thm:random} we have that a.a.s.~$D$ is an $(n,p,\kappa,\lambda)$-pseudorandom digraph, where $\kappa = \sqrt{np(1-p)}/(155\log^{3/4}n)$ and $\lambda = 5\sqrt{n/(1-p)}\log^2n$.
As mentioned in \cref{remark:newpseudorandom}, $D$ is also an $(n,p,\kappa', \lambda')$-pseudorandom digraph for any $\kappa' \leq \kappa$ and any $\lambda' \geq \lambda$.
Taking $\kappa' = 6(N^2p)^{1/5}$ and $\lambda' = \min\{np/3, (\kappa')^2 / 12\}$, and checking that $\kappa' < \kappa$ and $\lambda' > \lambda$ for the given range of $p$
\COMMENT{
Note that the first inequality is equivalent after rearrangement to $C\log^{23/6} n n^{-1/3} \leq p(1-p)^{5/3}$ for a suitable constant $C$. If $p< 1/2$ then the RHS is at least $(n^{-1/3}\log^4 n)/4$ so the inequality is satisfied, and if $p\geq 1/2$ then the RHS is at least $(n^{-1/3} \log^{25/6} n)/2$ so the inequality is satisfied.
\\
For the second inequality, we must show  that 
\[
\frac{5n^{1/2} \log^2 n}{(1-p)^{1/2}} \leq np/3 \:\:\: \text{ and } \:\:\:
\frac{5n^{1/2} \log^2 n}{(1-p)^{1/2}} \leq \kappa'^2/12 = 100(N^{2}p)^{2/5}/12.
\]
After rearrangement, the first of these is equivalent to $Cn^{-1/2}\log^2 n \leq p(1-p)^{1/2}$, which holds in our range of $p$. After rearrangement, the second of these is equivalent  to $Cn^{-3/10} \log^{6/5} n \leq p^{2/5}(1-p)^{1/2}$, which also holds in our range of $p$.
}
and $n$ sufficiently large, we have that~$D$ is an $(n,p,\kappa', \lambda')$-pseudorandom digraph, so we can apply \cref{thm:main} to conclude that~$D$ has a perfect decomposition (that is, it is consistent).
\end{proof}

In order to prove \cref{thm:random}, we will show that each of the properties of \cref{def:digraphclass1} holds a.a.s.
First, we require some properties about the edge distribution in $D_{n,p}$.

\begin{lemma}\label{lem:betaedges}
There exists a constant $C>0$ such that, for all $p\geq C\log n/n$, a.a.s.~the digraph $D=D_{n,p}$ satisfies that, for all $A\subseteq V(D)$ with $|A|\geq\log n/(50p)$, we have 
\[e_D(A)<100|A|^2p.\]
\end{lemma}

\begin{proof}
Fix some $\log n/(50p)\leq i\leq n$, and fix a set $A\subseteq V(D)$ with $|A|=i$.
Let $X\coloneqq e_D(A)$, so $\mathbb{E}[X]=(1-1/i)i^2p$.
A direct application of \cref{lem:betaChernoff} shows that, for sufficiently large $n$,\COMMENT{Since $i^2p\to\infty$ with $n$, we have that
\begin{align*}
    \mathbb{P}[X\geq100i^2p]\leq\mathbb{P}[X\geq100\mathbb{E}[X]]\leq(e/100)^{100\mathbb{E}[X]}=(e/100)^{100(1-1/i)i^2p}\leq(e/100)^{99i^2p},
\end{align*}
where the last inequality holds since $1-1/i=1-o(1)$.}
\[\mathbb{P}[X\geq100i^2p]\leq(e/100)^{99i^2p}.\]
Now consider all sets $A$ with $|A|=i$, and let $\mathcal{E}_i$ be the event that at least one of these sets induces at least $100i^2p$ edges.
By a union bound, it follows that\COMMENT{In order to see the last inequality, let us take logarithms.
Clearly, the inequality is equivalent to
\[i(1+\log n-\log i)+99i^2p(1-\log100)\leq-3\log n\iff i(1+\log n-\log i)+3\log n\leq99i^2p(\log100-1).\]
Now clearly, if $i$ is sufficiently large, $i(1+\log n-\log i)+3\log n\leq2i\log n$ and $99i^2p(\log100-1)\geq100i^2p$, so it would suffice to check that
\[2i\log n\leq100i^2p,\]
and this holds by the bound on $i$ in the statement.}
\[\mathbb{P}[\mathcal{E}_i]\leq\binom{n}{i}\left(\frac{e}{100}\right)^{99i^2p}\leq\left(\frac{en}{i}\right)^{i}\left(\frac{e}{100}\right)^{99i^2p}\leq\frac{1}{n^3},\]
where one can check the final inequality using the lower bound on $i$.
The conclusion follows by a union bound over all values of $i$.
\end{proof}

\begin{lemma}\label{lem:degrees}
There exist constants $C,c>0$ such that, for all $C\log n/n\leq p\leq 1-C\log n/n$, with probability at least $1-o(1/n^3)$ the digraph $D=D_{n,p}$ satisfies that, for all $v\in V$, we have 
\[d_D^+(v)=np\pm c\sqrt{np(1-p)\log n}\qquad\text{ and }\qquad d_D^-(v)=np\pm c\sqrt{np(1-p)\log n}.\]
\end{lemma}

\begin{proof}
We split the proof into two cases.
Assume first that $p\leq1/2$.
Fix a vertex $v\in V(D)$ and a symbol $*\in\{+,-\}$.
Then, $\mathbb{E}[d_D^*(v)]=(n-1)p$ and, if $C$ and $n$ are sufficiently large (we need $C$ to be sufficiently large so that the value of $\delta$ in \cref{lem:Chernoff} satisfies $\delta\in(0,1)$), by \cref{lem:Chernoff} we conclude that\COMMENT{Let $X\coloneqq d_D^*(v)\sim\mathrm{Bin}(n-1,p)$.
We have
\begin{align*}
    \mathbb{P}\left[X\neq np\pm c\sqrt{np(1-p)\log n}\right]&\leq\mathbb{P}\left[X\neq(n-1)p\pm c\sqrt{(n-1)p(1-p)\log n}/2\right]\\
    &\leq\mathbb{P}\left[X\neq(n-1)p\pm c\sqrt{(n-1)p\log n}/4\right]\\
    &=\mathbb{P}\left[X\neq\left(1\pm\frac{c}{4}\sqrt{\frac{\log n}{(n-1)p}}\right)(n-1)p\right]\\
    &\leq2e^{-\frac{c^2}{16}\frac{\log n}{(n-1)p}(n-1)p/3}\leq e^{-c^2\log n/50}.
\end{align*}
For the second inequality we are using the fact that $1-p\geq1/2$.
The last inequality holds for $n$ sufficiently large.
Observe that the condition that $p\geq C\log n/n$ is needed to guarantee that the $\delta$ with which we apply \cref{lem:Chernoff} lies in $(0,1)$.
Indeed, in our application of the Chernoff bound we have
\[\delta=\frac{c}{4}\sqrt{\frac{\log n}{(n-1)p}}<1\iff p>\frac{c^2}{16}\frac{\log n}{n-1},\]
so it suffices to have $p>c^2\log n/8n$.}
\[\mathbb{P}\left[d_D^*(v)\neq np\pm c\sqrt{np(1-p)\log n}\right]\leq e^{-c^2\log n/50}.\]
Now, by a union bound over all choices of $v$ and $*$, it follows that the probability that the statement fails is at most\COMMENT{$2ne^{-c^2\log n/50}=e^{\log 2+(1-c^2/50)\log n}=o(1/n^3)$} $2ne^{-c^2\log n/50}=o(1/n^3)$ (where this equality holds for sufficiently large $c$).\COMMENT{$c\geq15$ suffices.}

For the second case, assume $p>1/2$, and consider the complement digraph $\overline{D}\sim D_{n,1-p}$.
We have that $1-p<1/2$, so we can apply the same argument as above to obtain that, for each $v\in V(D)$ and $*\in\{+,-\}$,
\[\mathbb{P}\left[d_{\overline{D}}^*(v)\neq n(1-p)\pm c\sqrt{np(1-p)\log n}\right]\leq e^{-c^2\log n/50}.\]
The conclusion follows by a union bound over all $v\in V(D)$ and $*\in\{+,-\}$ and going back to $D$\COMMENT{We have that $d^*_D(v)=n-1-d^*_{\overline{D}}=n-1-n(1-p)\pm c\sqrt{np(1-p)\log n}=np-1\pm c\sqrt{np(1-p)\log n}$, and this is what we want (by making $c$ slightly worse).}.
\end{proof}

Our next aim is to show that most vertices will have ``high'' excess, meaning that its absolute value is ``close'' to the maximum possible value (around $\sqrt{np(1-p)}$, up to a polylog factor) that follows from \cref{lem:degrees}.
The following remark will come in useful.

\begin{remark}\label{rem:exprobdist}
Let $p\in[0,1]$ and $n\in\mathbb{Z}$ with $n\geq0$. 
Let $X\sim\mathrm{Bin}(n,p)$.
For each $i\in\mathbb{Z}$, let $p_i\coloneqq\mathbb{P}[X=i]$.
Let $D$ be a digraph and $v\in V(D)$ be such that $d^+(v)=d^-(v)=n$.
Let $D_p$ be a random subdigraph of $D$ obtained by deleting each edge of $D$ with probability $1-p$ independently of all other edges.
Then, $\ex_{D_p}(v)$ follows a probability distribution which, for each $i\in\{-n,\ldots,n\}$, satisfies that
\[\mathbb{P}[\ex_{D_p}(v)=i]=\sum_{j=0}^{n}p_jp_{j-i}.\]
In particular, the probability function is symmetric around $i=0$.
\end{remark}

\begin{lemma}\label{lem:q0bound}
Consider the setting described in \cref{rem:exprobdist}, and assume $n\geq2$.
Then, there exists an absolute constant $K$ such that
\[\mathbb{P}[\ex_{D_p}(v)=0]\leq K\sqrt{\frac{\log n}{np(1-p)}}.\]
\end{lemma}

\begin{proof}
First note that, by adjusting the value of $K$, we may assume that $n$ is larger than any fixed $n_0$ (by making the right hand side above greater than $1$)\COMMENT{For every $n\geq1$ we have that $\log n/(np(1-p))\geq\log n/n$. Simply, for any given $n_0$, we can set $K\geq\sqrt{n_0/\log n_0}$ (note this function is increasing for $n_0\geq3$, and soon overtakes the value of $n_0=2$), which means that, for $n\leq n_0$, the statement is satisfied by the trivial upper bound:
\[\mathbb{P}[\ex_{D_p}(v)=0]\leq1=\sqrt{\frac{n}{\log n}\frac{\log n}{n}}\leq\sqrt{\frac{n_0}{\log n_0}\frac{\log n}{n}}\leq K\sqrt{\frac{\log n}{n}}\leq K\sqrt{\frac{\log n}{np(1-p)}}.\]
}; we choose a sufficiently large $n_0$ so that all subsequent claims hold.
By similarly adjusting the value of $K$, for any given constant $C_0$ we may assume that $C_0\log n/n\leq p\leq1-C_0\log n/n$\COMMENT{Indeed, if we assume $p<C_0\log n/n$, by adjusting $K$ we may guarantee that
\[K\sqrt{\frac{\log n}{np(1-p)}}>K/\sqrt{C_0}\geq1,\]
and the case when $p>1-C_0\log n/n$ is proved analogously.}.

So assume $C\log n/n\leq p\leq1-C\log n/n$, for a constant $C$ defined below.
One can readily check that $p^*\coloneqq\max_{i\in[n]_0}p_i$ is achieved for $i=np\pm2$\COMMENT{Consider the ratio $p_{i+1}/p_i$, for $i\in[n-1]_0$.
We have that
\[\frac{p_{i+1}}{p_i}=\frac{\displaystyle\binom{n}{i+1}p^{i+1}(1-p)^{n-i-1}}{\displaystyle\binom{n}{i}p^i(1-p)^{n-i}}=\frac{i!(n-i)!}{(i+1)!(n-i-1)!}\frac{p}{1-p}=\frac{n-i}{i+1}\frac{p}{1-p}.\]
We want to know when this ratio changes from greater than $1$ (which means the ratio is increasing) to when it is less than $1$ (decreasing).
By setting $p_{i+1}/p_i=1$ and isolating, we have that
\begin{align*}
    \frac{p_{i+1}}{p_i}=1&\iff\frac{n-i}{i+1}\frac{p}{1-p}=1\iff(n-i)\frac{p}{1-p}=i+1\\
    &\iff\frac{np}{1-p}=i+1+\frac{p}{1-p}i=\frac{i}{1-p}+1\iff i=np-1+p.
\end{align*}
Since this is the only solution, we know that the maximum must be achieved for either $\lfloor np-1+p\rfloor$ or $\lceil np-1+p\rceil$, and both of these lie in $np\pm2$.} (where the $p_i$ are as defined in \cref{rem:exprobdist}).
By using Stirling's approximation, it follows that $p^*\leq1/\sqrt{np(1-p)}$\COMMENT{Let us write $p_{np}$, and assume $n$ is sufficiently large (smaller values of $n$ are hidden by $K$; it is easy to check that, for the given range of $p$, $\pm2$ does not affect the asymptotic statements, and we will increase the final constant here to avoid issues).
We have that 
\[p_{np}=\binom{n}{np}p^{np}(1-p)^{n(1-p)}.\]
We now use the bounds $\sqrt{2\pi}n^{n+1/2}e^{-n}\leq n!\leq en^{n+1/2}e^{-n}$, similar to Stirling's approximation and valid for all $n$, to conclude that
\begin{align*}
    p_{np}\leq\frac{\displaystyle e\sqrt{n}\left(\frac{n}{e}\right)^n}{\displaystyle \sqrt{2\pi}\sqrt{np}\left(\frac{np}{e}\right)^{np}\sqrt{2\pi}\sqrt{(1-p)n}\left(\frac{(1-p)n}{e}\right)^{(1-p)n}}p^{np}(1-p)^{(1-p)n}=\frac{e}{2\pi}\frac{1}{\sqrt{np(1-p)}}.
\end{align*}
Clearly the constant in front is less than $1$, so what we claim must hold true by considering the small changes by $\pm2$.}.
On the other hand, by an application of \cref{lem:Chernoff}, there exist constants $c,C>0$ such that for all $C\log n/n\leq p\leq1-C\log n/n$ we have that\COMMENT{This sum of probabilities is the same as the probability that the outcome of a binomial variable deviates from its mean by at least $c\sqrt{np(1-p)\log n}$.
Check the proof of \cref{lem:degrees} for the details of the calculation.}
\[\sum_{i=0}^{np-c\sqrt{np(1-p)\log n}}p_i+\sum_{i=np+c\sqrt{np(1-p)\log n}}^{n}p_i\leq e^{-c^2\log n/50}.\]
Combining the above with \cref{rem:exprobdist}, it follows that\COMMENT{We have that
\begin{align*}
    \mathbb{P}[\ex_D(v)=0]&=\sum_{i=0}^{n-1}p_i^2=\sum_{i=0}^{np-c\sqrt{np(1-p)\log n}}p_i^2+\sum_{i=np-c\sqrt{np(1-p)\log n}}^{np+c\sqrt{np(1-p)\log n}}p_i^2+\sum_{i=np+c\sqrt{np(1-p)\log n}}^{n-1}p_i^2\\
    &\leq\sum_{i=0}^{np-c\sqrt{np(1-p)\log n}}p_i+2c\sqrt{np(1-p)\log n}(\max_{i\in[n]_0}p_i)^2+\sum_{i=np+c\sqrt{np(1-p)\log n}}^{n-1}p_i\\
    &\leq e^{-c^2\log n/50}+2c\sqrt{np(1-p)\log n}/(np(1-p))\leq K\sqrt{\frac{\log n}{np(1-p)}}.
\end{align*}
Note that we need $c$ and $C$ to be large enough so that we can apply \cref{lem:Chernoff}, and also $c$ to be large enough so that the second term above dominates.}
\COMNEW{In fact, we should be able to show that $\mathbb{P}[\ex_{D_p}(v)=0]=\Theta(1/\sqrt{np(1-p)})$.
I think this is true by approximating with integrals (see, e.g., here: https://math.stackexchange.com/questions/2177215/asymptotic-probability-that-two-binomial-variables-are-equal), but I wonder if we can do something better. In any case, we do not care right now, but this will be important if we want to prove for the whole range of $p$.}
\[
    \mathbb{P}[\ex_{D_p}(v)=0]\leq2c\sqrt{np(1-p)\log n}\cdot (p^*)^2+e^{-c^2\log n/50}\leq K\sqrt{\frac{\log n}{np(1-p)}}.\qedhere
\]
\end{proof}

\begin{lemma}\label{coro:excessbound2}
There exists a constant $C>0$ such that, for all $C\log n/n\leq p\leq 1-C\log n/n$, a.a.s.~the digraph $D=D_{n,p}$ contains at most $n/\log^{1/8}n$ vertices $v$ such that $|\ex_D(v)|\leq\sqrt{np(1-p)}/\log^{3/4}n$\COMMENT{The power in this log can be brought arbitrarily close to $1/2$.}\COMNEW{Actually, we can prove that only $o(n)$ vertices have excess $o(\sqrt{np})$ (this requires getting rid of the logarithmic factor in \cref{lem:q0bound}, which can be done). But perhaps it is nicer to have an explicit function.}.
\end{lemma}

\begin{proof}
Take some vertex $v\in V(D)$.
For each $i\in\mathbb{Z}$, let $p_i\coloneqq\mathbb{P}[d_D^+(v)=i]=\mathbb{P}[d_D^-(v)=i]$.
Now, by \cref{rem:exprobdist} we have that $q_0\coloneqq\mathbb{P}[\ex_D(v)=0]=\sum_{j=0}^{n-1}p_j^2$ and, for all $i\in[n-1]$, we have that $q_i\coloneqq\mathbb{P}[|\ex_D(v)|=i]=\sum_{j=0}^{n-1}p_j(p_{j-i}+p_{j+i})$.
In particular, by \cref{lem:rearrangement}, it follows that
\begin{equation}\label{equa:q0rear}
    q_i\leq 2q_0
\end{equation}
for all $i\in[n-1]$\COMMENT{We apply the second inequality to our setting, where both lists have the same elements and have (say) $5n$ elements (from $p_{-2n}$ to $p_{3n}$), to guarantee that everything is covered.}.
By combining this with \cref{lem:q0bound} (with $n-1$ playing the role of $n$), it follows that
\begin{equation}\label{eq:boundexcess}
    \mathbb{P}[|\ex_D(v)|\leq\sqrt{np(1-p)}/\log^{3/4}n]=\bigO(1/\log^{1/4}n).
\end{equation}
Let $Y\coloneqq|\{v\in V(D):|\ex_D(v)|\leq\sqrt{np(1-p)}/\log^{3/4}n\}|$.
The statement follows by applying Markov's inequality to this random variable\COMMENT{We have that $\mathbb{E}[Y]=\bigO(n/\log^{1/4}n)$.
By Markov's inequality, it follows that 
\[\mathbb{P}[Y\geq n/\log^{1/8}n]\leq\mathbb{E}[Y]\log^{1/8}n/n=\bigO(1/\log^{1/8}n)=o(1).\]}.
\end{proof}

We consider a partition of the vertices of $D_{n,p}$ into those with high excess, low excess, and the rest.
In general, given $D=D_{n,p}$, we write 
\begin{align*}
    A^+&=A^+(D)\coloneqq\{v\in V(D):\ex_D(v)\geq\sqrt{np(1-p)}/\log^{3/4}n\},\\
    A^-&=A^-(D)\coloneqq\{v\in V(D):\ex_D(v)\leq-\sqrt{np(1-p)}/\log^{3/4}n\}\text{ and}\\
    A^0&=A^0(D)\coloneqq V(D)\setminus(A^+\cup A^-).
\end{align*}
\Cref{coro:excessbound2} shows that $|A^0|=o(n)$, and it is reasonable to expect that $A^+$ and $A^-$ have roughly the same size.
Even more, we will need the property that, with high probability, all vertices have roughly the expected number of neighbours in the sets $A^+$ and $A^-$, as we show next.
The proof, while rather technical, follows standard martingale arguments and does not really require any new insights.
%We were unable to find a simpler proof. 

\begin{lemma}\label{lemma:A0edges}
There exists a constant $C>0$ such that, for all $C\log n/n\leq p\leq 1-C\log n/n$\COMNEW{Note that the result is only meaningful if $\log^2n/\sqrt{n}\ll p\ll 1-\log^4n/n$; otherwise, the error term is much larger than the main term.}, a.a.s.~the graph $D=D_{n,p}$ satisfies that, for all $v\in V(D)$,
\[e(v,A^+),e(v,A^-),e(A^+,v),e(A^-,v)= np/2\pm2\sqrt{n/(1-p)}\log^2n.\]
\end{lemma}

\begin{proof}
Let $V\coloneqq V(D)$, and let $E\coloneqq\{uv:u,v\in V, u\neq v\}$.
Let $N\coloneqq\binom{n}{2}=|E|/2$.
For each $k\in[n-1]_0$, let $Z_k\sim\mathrm{Bin}(k,p)$ and, for each $j\in\mathbb{Z}$, let $p_j^{(k)}\coloneqq\mathbb{P}[Z_k=j]$.

We begin by setting some notation.
Consider any labelling $e_1,\ldots,e_N$ of all (unordered) pairs of distinct vertices $e=\{u,u'\}$ with $u,u'\in V$.
We will later reveal the edges in succession following one such labelling.
For each $i\in[N]$, let $e_i=\{u_i,u_i'\}$, define $e_i^1\coloneqq u_iu_i'$ and $e_i^2\coloneqq u_i'u_i$ (the choice of $e_i^1$ and $e_i^2$ is arbitrary), and consider the random variable $X_i\coloneqq(X_i^1,X_i^2)$, where $X_i^1$ and $X_i^2$ are indicator random variables for the events $\{e_i^1\in E(D)\}$ and $\{e_i^2\in E(D)\}$, respectively.
For each $i\in[N]_0$, let $D^i\coloneqq(V,E^i)$, where $E^i\coloneqq\bigcup_{j\in[i]}\{e_j^1,e_j^2\}$.
We set $D^i_{\mathrm{cond}}\coloneqq(V,E^i_\mathrm{cond})$ to be the subdigraph of $D^i$ with $E^i_\mathrm{cond}\coloneqq\{e_j^1:j\in[i],X_j^1=1\}\cup\{e_j^2:j\in[i],X_j^2=1\}$.
(That is, without conditioning, $D^i_{\mathrm{cond}}$ is a random subdigraph of $D^i$ where each edge is retained with probability $p$ independently of all other edges, and it becomes a deterministic graph after conditioning on the outcomes of $X_1,\ldots,X_i$.)
We also define $D^i_{p}\coloneqq(V,E_{i,p})$, where $E_{i,p}\subseteq E\setminus{E^i}$ is obtained by adding each edge of $E\setminus{E^i}$ with probability $p$, independently of all other edges.
In particular, for any $i\in[N]_0$ and any digraph $F$ on $V$ such that $D^i_\mathrm{cond}\subseteq F$, we have that $\mathbb{P}[D_{n,p}=F\mid X_1,\ldots,X_i]=\mathbb{P}[D^i_p=F\setminus D^i_\mathrm{cond}]$.
For each $i\in[N]_0$ and each $v\in V$, we define $k_i(v)\coloneqq n-1-|\{u\in V:uv\in E^i\}|$. 
This is the number of (pairs of) edges incident to~$v$ which have not been revealed after revealing $X_1,\ldots,X_i$.
Thus, by \cref{rem:exprobdist}, the variable $\ex_{D^i_p}(v)$ follows a probability distribution which, for each $j\in\mathbb{Z}$, satisfies that
\begin{equation}\label{equa:A+lemma1distrib}
    \mathbb{P}[\ex_{D^{i}_p}(v)=j]=\sum_{\ell=0}^{k_{i}(v)}p^{(k_{i}(v))}_\ell p^{(k_{i}(v))}_{\ell-j}.
\end{equation}
Observe that, by \cref{lem:rearrangement} (in a similar way to \eqref{equa:q0rear}), for all $i\in[N]_0$, $v\in V$ and $j\in\mathbb{Z}$ we have that
\begin{equation}\label{equa:A+lemma1distribbound}
    \mathbb{P}[\ex_{D^{i}_p}(v)=j]\leq q_0^{(k_i(v))}\coloneqq\mathbb{P}[\ex_{D^{i}_p}(v)=0]=\sum_{\ell=0}^{k_{i}(v)}\left(p^{(k_{i}(v))}_\ell\right)^2.
\end{equation}
Furthermore, observe the following.
Choose a vertex $v\in V$ and an index $i\in[N-1]_0$ such that $d^+_{D^{i+1}}(v)-d^+_{D^i}(v)=1$, and let $a\in\mathbb{Z}$.
Then,\COMMENT{Indeed, consider the joint distribution of $D^i_p$ and $D^{i+1}_p$ where we first reveal $D^i_p$ and then reveal the last pair of edges needed to obtain $D^{i+1}_p$.
If $\ex_{D^{i+1}_p}(v)\geq a$, then we are guaranteed that $\ex_{D^i_p}(v)\geq a-1$, since the last pair of edges we reveal can only increase the excess by at most $1$.
Similarly, we cannot have $\ex_{D^i_p}(v)\geq a+1$ unless $\ex_{D^{i+1}_p}(v)\geq a$, as the excess cannot decrease by more than $1$ when revealing the last pair of edges.}
\COMNEW{About this inequality down here, and why it is phrased this way, with the conditioning written the way it is even though the events are independent from the conditions.
Indeed, what we have is that 
\[\mathbb{P}[\ex_{D^i_p}(v)\geq a+1]\leq\mathbb{P}[\ex_{D^{i+1}_p}(v)\geq a]\leq\mathbb{P}[\ex_{D^i_p}(v)\geq a-1].\]
But the ``easy'' way to see why this is true is to consider the joint distribution of these two variables where we reveal edges one by one, which implies the given bounds (see the previous comment).
By writing the conditioning explicitly, I emphasise that the variables are the same (i.e. I am assuming that the variables $X_1,\ldots,X_i$ take the same values on all sides of the inequality).\\
Ultimately, I am just not sure how to write this, and I would be happy with any other opinions; I am actually not that good at writing technical probabilistic statements, and I find that they can be quite confusing for a non-expert.}
\begin{align*}
    \mathbb{P}[\ex_{D^i_p}(v)\geq a+1\mid X_1,\ldots,X_i]&\leq\mathbb{P}[\ex_{D^{i+1}_p}(v)\geq a\mid X_1,\ldots,X_{i+1}]\\
    &\leq\mathbb{P}[\ex_{D^i_p}(v)\geq a-1\mid X_1,\ldots,X_i].
\end{align*}
(Note that the events above are actually independent from the variables upon which we condition.
This notation, however, makes the statement more intuitive and is what we will require later in the proof.)
In particular, this means that\COMMENT{Indeed, we have that either $\mathbb{P}[\ex_{D^i_p}(v)\geq a+1\mid X_1,\ldots,X_i]\leq\mathbb{P}[\ex_{D^{i+1}_p}(v)\geq a\mid X_1,\ldots,X_{i+1}]\leq\mathbb{P}[\ex_{D^i_p}(v)\geq a\mid X_1,\ldots,X_i]$ or $\mathbb{P}[\ex_{D^i_p}(v)\geq a\mid X_1,\ldots,X_i]\leq\mathbb{P}[\ex_{D^{i+1}_p}(v)\geq a\mid X_1,\ldots,X_{i+1}]\leq\mathbb{P}[\ex_{D^i_p}(v)\geq a-1\mid X_1,\ldots,X_i]$.
Now assume that the first of the two cases holds (the other follows similarly).
The claim follows since $\mathbb{P}[\ex_{D^i_p}(v)\geq a\mid X_1,\ldots,X_i]-\mathbb{P}[\ex_{D^i_p}(v)\geq a+1\mid X_1,\ldots,X_i]\leq q_0^{(k_i(v))}$.\\
If this is a bit hard to follow, we can add a factor of $2$; it does not change the asymptotics.}
\begin{equation}\label{equa:A+lemma1nicebound}
    \left\lvert\mathbb{P}[\ex_{D^i_p}(v)\geq a\mid X_1,\ldots,X_i]-\mathbb{P}[\ex_{D^{i+1}_p}(v)\geq a\mid X_1,\ldots,X_{i+1}]\right\rvert\leq q_0^{(k_i(v))}.
\end{equation}
(Indeed, we may bound $\mathbb{P}[\ex_{D^{i+1}_p}(v)\geq a\mid X_1,\ldots,X_{i+1}]$ by one of the two terms in the previous expression, which gives us two cases to consider.
In either of the cases, the difference becomes  equal to the probability that $\ex_{D^i_p}(v)$ takes a specific value, which is in turn bounded by \eqref{equa:A+lemma1distribbound}.)

Fix a vertex $v\in V$ and reveal all of its in- and outneighbours.
Label all pairs of distinct vertices $e$ as $e_1,\ldots,e_N$ in such a way that, first, we have all pairs containing $v$, and then the rest, in any arbitrary order.
In particular, we have already revealed the outcome of $X_1,\ldots,X_{n-1}$.
Let $\mathcal{E}$ be the event that $d_D^+(v),d_D^-(v)=np\pm c\sqrt{np(1-p)\log n}$, where $c$ is the constant from the statement of \cref{lem:degrees}.
By \cref{lem:degrees}, we have that $\mathbb{P}[\mathcal{E}]\geq1-1/n^3$.
Condition on this event.
We will denote probabilities in this conditional space by $\mathbb{P}'$, and expectations by $\mathbb{E}'$.
Observe that the variables $X_n,\ldots,X_N$ are independent of $\mathcal{E}$\COMMENT{So for all events that only involve these variables we have that $\mathbb{P}'=\mathbb{P}$.}.
Then, for all $u\in N^+_D(v)$ we have that $\ex_D(u)=\ex_{D^{n-1}_\mathrm{cond}}(u)+\ex_{D^{n-1}_p}(u)$, where $\ex_{D^{n-1}_\mathrm{cond}}(u)=0$ if $u\in N^-_D(v)$ and $\ex_{D^{n-1}_\mathrm{cond}}(u)=-1$ otherwise, and $\ex_{D^{n-1}_p}(u)$ follows a probability distribution which, by \eqref{equa:A+lemma1distrib}, for each $j\in\{2-n,\ldots,n-2\}$ satisfies that
\[
    \mathbb{P}'[\ex_{D^{n-1}_p}(u)=j]=\sum_{\ell=0}^{n-2}p^{(n-2)}_\ell p^{(n-2)}_{\ell-j}.
\]
By a similar argument as the one used to obtain \eqref{eq:boundexcess}, i.e., combining \eqref{equa:A+lemma1distribbound} and the above with \cref{lem:q0bound} (with $n-2$ playing the role of $n$), it follows that, for all $u\in V\setminus\{v\}$\COMMENT{By \eqref{equa:A+lemma1distribbound}, each $\mathbb{P}'[\ex_{D^{n-1}_p}(u)=i]$ (recall that here $\mathbb{P}=\mathbb{P}'$) is at most $q_0^{(n-2)}$, so
\begin{align*}
\mathbb{P}'[|\ex_{D^{n-1}_p}(u)|\geq\sqrt{np(1-p)}/\log^{3/4}n]&\geq\mathbb{P}'[|\ex_{D^{n-1}_p}(u)|>\sqrt{np(1-p)}/\log^{3/4}n]\\
&=1-\mathbb{P}'[|\ex_{D^{n-1}_p}(u)|\leq\sqrt{np(1-p)}/\log^{3/4}n]\\
&=1-\sum_{i=-\sqrt{np(1-p)}/\log^{3/4}n}^{\sqrt{np(1-p)}/\log^{3/4}n}\mathbb{P}[\ex_{D^{n-1}_p}(u)=i]\\
&\geq 1-2\sqrt{np(1-p)}/\log^{3/4}n\cdot q_0^{(n-2)}=1-\bigO(1/\log^{1/4}n),
\end{align*}
where the last inequality uses \cref{lem:q0bound} (the change from $n$ to $n-2$ does not change the asymptotics).},
\[\mathbb{P}'[|\ex_{D^{n-1}_p}(u)|\geq\sqrt{np(1-p)}/\log^{3/4}n]=1-\bigO(1/\log^{1/4}n)\]
and, therefore, one easily deduces (by symmetry and conditioning on the event that $|\ex_{D^{n-1}_p}(u)|\geq\sqrt{np(1-p)}/\log^{3/4}n$) that\COMMENT{To see this, observe the following.
Let $\mathcal{E}_1$ be the event that $|\ex_D(u)|\geq\sqrt{np(1-p)}/\log^{3/4}n$.
Then, we have that $\mathbb{P}'[u\in A^+]=\mathbb{P}'[\mathcal{E}_1]\,\mathbb{P}'[\ex_D(u)>0\mid\mathcal{E}_1]$.
We have that $\mathbb{P}'[\mathcal{E}_1]=1-\bigO(1/\log^{1/4}n)$ (indeed, by the discussion above we have that $\ex_D(u)=\ex_{D^{n-1}_p}(u)\pm1$, and $\mathbb{P}'[|\ex_D(u)|\geq\sqrt{np(1-p)}/\log^{3/4}n]\geq\mathbb{P}'[|\ex_{D^{n-1}_p}(u)|\geq2\sqrt{np(1-p)}/\log^{3/4}n]=1-\bigO(1/\log^{1/4}n)$).
After conditioning on this, the vertex has positive or negative excess only depending on whether $\ex_{D^{n-1}_p}(u)$ is positive or negative, and the probability distribution for this random variable is symmetric around $0$, so each must have probability a half.}
\begin{equation}\label{equa:A+lemma1size}
    \mathbb{P}'[u\in A^+]=1/2-\bigO(1/\log^{1/4}n).
\end{equation}

Consider the edge-exposure martingale given by the variables $Y_i\coloneqq\mathbb{E}'[|A^+\cap N^+_D(v)|\mid X_1,\ldots,X_i]$, for $i\in[N]\setminus[n-2]$.
By \eqref{equa:A+lemma1size}, it follows that $Y_{n-1},\ldots,Y_N$ is a Doob martingale with $Y_{n-1}=\mathbb{E}'[|A^+\cap N^+_D(v)|]=(1\pm 2c\sqrt{(1-p)\log n/(np)}-\bigO(1/\log^{1/4}n))np/2$ and $Y_N=|A^+\cap N^+_D(v)|$.
In order to prove that $Y_N$ is concentrated around $Y_{n-1}$, we need to bound the martingale differences with a view to applying \cref{lem:Azuma}.
Observe that, for all $i\in[N]\setminus[n-2]$, we have that $Y_i=\sum_{u\in N^+_D(v)}\mathbb{P}'[u\in A^+\mid X_1,\ldots,X_i]$.

For all $i\in[N-1]\setminus[n-2]$ such that $e_{i+1}\cap N^+_D(v)=\varnothing$, we have that $Y_{i+1}=Y_i$, and we set 
\begin{equation}\label{equa:A+lemma1ci1}
    c_i\coloneqq|Y_{i+1}-Y_i|=0.
\end{equation}
Consider now any $i\in[N-1]\setminus[n-2]$ such that $e_{i+1}=\{u,u'\}$ satisfies that $e_{i+1}\cap N^+_D(v)=\{u\}$.
Then, 
\begin{align*}
    Y_{i+1}-Y_i&=\mathbb{P}'[u\in A^+\mid X_1,\ldots,X_{i+1}]-\mathbb{P}'[u\in A^+\mid X_1,\ldots,X_i]\\
    &=\mathbb{P}'[\ex_D(u)\geq\sqrt{np(1-p)}/\log^{3/4}n\mid X_1,\ldots,X_{i+1}]\\
    &\ \ \ \qquad\qquad-\mathbb{P}'[\ex_D(u)\geq\sqrt{np(1-p)}/\log^{3/4}n\mid X_1,\ldots,X_i]\\
    &=\mathbb{P}'[\ex_{D^{i+1}_p}(u)\geq\sqrt{np(1-p)}/\log^{3/4}n-\ex_{D^{i+1}_\mathrm{cond}}(u)\mid X_1,\ldots,X_{i+1}]\\
    &\ \ \ \qquad\qquad-\mathbb{P}'[\ex_{D^i_p}(u)\geq\sqrt{np(1-p)}/\log^{3/4}n-\ex_{D^i_\mathrm{cond}}(u)\mid X_1,\ldots,X_i],
\end{align*}
so by \eqref{equa:A+lemma1distribbound} and \eqref{equa:A+lemma1nicebound}, and using the fact that $|\ex_{D^{i+1}_\mathrm{cond}}(u)-\ex_{D^i_\mathrm{cond}}(u)|\leq1$, we conclude that\COMMENT{Here we use the fact that $\ex_{D^{i+1}_\mathrm{cond}}(u)=\ex_{D^i_\mathrm{cond}}(u)\pm1$.
Thus, we have three cases.
Let us show here the case when $\ex_{D^{i+1}_\mathrm{cond}}(u)=\ex_{D^i_\mathrm{cond}}(u)+1$; the other two are done similarly.
Let $a\coloneqq\sqrt{np(1-p)}/\log^{3/4}n-\ex_{D^i_\mathrm{cond}}(u)$.
Using \eqref{equa:A+lemma1nicebound}, we conclude that 
\begin{align*}
    &|\mathbb{P}'[\ex_{D^{i+1}_p}(u)\geq a-1\mid X_1,\ldots,X_{i+1}]-\mathbb{P}'[\ex_{D^i_p}(u)\geq a\mid X_1,\ldots,X_i]|\\
    \leq &|\mathbb{P}'[\ex_{D^{i+1}_p}(u)\geq a-1\mid X_1,\ldots,X_{i+1}]-\mathbb{P}'[\ex_{D^i_p}(u)\geq a-1\mid X_1,\ldots,X_i]|\\
    &+|\mathbb{P}'[\ex_{D^i_p}(u)\geq a-1\mid X_1,\ldots,X_i]-\mathbb{P}'[\ex_{D^i_p}(u)\geq a\mid X_1,\ldots,X_i]|\\
    \leq&2q_0^{(k_i(u))}.
\end{align*}
Here, the first inequality holds by the triangle inequality.
Then, the first difference is bounded by \eqref{equa:A+lemma1nicebound}, and the second, by \eqref{equa:A+lemma1distribbound}.}
\begin{equation}\label{equa:A+lemma1ci2}
    |Y_{i+1}-Y_i|\leq 2q_0^{(k_i(u))}\eqqcolon c_i.
\end{equation}
Finally, for any $i\in[N-1]\setminus[n-2]$ such that $e_{i+1}=\{u,u'\}\subseteq N^+_D(v)$, one can similarly show that
\begin{equation}\label{equa:A+lemma1ci3}
    |Y_{i+1}-Y_i|\leq 2(q_0^{(k_i(u))}+q_0^{(k_i(u'))})\eqqcolon c_i.
\end{equation}
This covers all the range of $i\in[N-1]\setminus[n-2]$.

By combining \eqref{equa:A+lemma1ci1}--\eqref{equa:A+lemma1ci3}, we observe that, for each $u\in N^+(v)$ and each $k\in[n-2]$, the value $q_0^{(k)}$ appears as part of $c_i$ for exactly one value of $i\in[N-1]\setminus[n-2]$.
Then, we have
\[\sum_{i=n-1}^{N-1}c_i^2\leq\sum_{u\in N_D^+(v)}\sum_{k=1}^{n-2}8\left(q_0^{(k)}\right)^2,\]
where we have used the fact that $(x+y)^2\leq2x^2+2y^2$.
Now, using \cref{lem:q0bound} and the conditioning on $\mathcal{E}$, we have that\COMMENT{In order to see the last equality, observe that, since $\frac{\log k}{kp(1-p)}$ is a function decreasing in $k$ for $k\geq3$ (we may treat $p$ as a constant, for this sum), we have that 
\[1+\sum_{k=2}^{n-2}\frac{\log k}{kp(1-p)}=1+\frac{1}{p(1-p)}\sum_{k=2}^{n-2}\frac{\log k}{k}\leq1+ \frac{1}{p(1-p)}\left(C+\int_3^{n-2}\frac{\log x}{x}\,\mathrm{d}x\right)=\bigO\left(\frac{\log^2n}{p(1-p)}\right),\]
where $C<1$ is the sum of the first two terms (which is a constant), and the last inequality follows since $\int_4^{n-2}\frac{\log x}{x}\,\mathrm{d}x=C_1+\log^2n/2$.}
\[
    \sum_{i=n-1}^{N-1}c_i^2\leq\left(1\pm c\sqrt{\frac{(1-p)\log n}{np}}\right)8K^2np\left(1+\sum_{k=2}^{n-2}\frac{\log k}{kp(1-p)}\right)=\bigO\left(\frac{n\log^2n}{(1-p)}\right).
\]
Therefore, we can apply \cref{lem:Azuma} to conclude that\COMMENT{We have that
\[\mathbb{P}'[|Y_N-Y_{n-1}|\geq t]\leq2e^{-\frac{t^2}{2\sum_{i=n-1}^{N-1}c_i^2}}=e^{-\Omega(t^2(1-p)/n\log^2n)}.\]
Thus, if $t=\sqrt{n/(1-p)}\log^2n$, we have that 
\[\mathbb{P}'[|Y_N-Y_{n-1}|\geq\sqrt{n/(1-p)}\log^2n]\leq e^{-\Omega(\log^2n)}.\]
The last bound follows by observing that $\mathbb{P}[|A^+\cap N_D^+(v)|\neq np/2\pm2\sqrt{n/(1-p)}\log^2n]\leq\mathbb{P}[|A^+\cap N_D^+(v)|\neq np/2\pm(c\sqrt{np(1-p)\log n}+\sqrt{n/(1-p)}\log^2n)]=\mathbb{P}[|Y_N-Y_{n-1}|\geq\sqrt{n/(1-p)}\log^2n]$.}
\begin{equation}\label{equa:A+lemma1bound1}
    \mathbb{P}'[|A^+\cap N_D^+(v)|\neq np/2\pm2\sqrt{n/(1-p)}\log^2n]=e^{-\Omega(\log^2n)}.
\end{equation}
By similar arguments\COMNEW{They are exactly the same, and very long to write properly...}, we can show that
\begin{align}
    \mathbb{P}'[|A^-\cap N_D^+(v)|\neq np/2\pm2\sqrt{n/(1-p)}\log^2n]=e^{-\Omega(\log^2n)},\label{equa:A+lemma1bound2}\\
    \mathbb{P}'[|A^+\cap N_D^-(v)|\neq np/2\pm2\sqrt{n/(1-p)}\log^2n]=e^{-\Omega(\log^2n)},\label{equa:A+lemma1bound3}\\
    \mathbb{P}'[|A^-\cap N_D^-(v)|\neq np/2\pm2\sqrt{n/(1-p)}\log^2n]=e^{-\Omega(\log^2n)}.\label{equa:A+lemma1bound4}
\end{align}

Let $\mathcal{E}'$ be the event that $|A^+\cap N_D^+(v)|,|A^-\cap N_D^+(v)|,|A^+\cap N_D^-(v)|,|A^-\cap N_D^-(v)|=np/2\pm2\sqrt{n/(1-p)}\log^2n$.
By combining \eqref{equa:A+lemma1bound1}--\eqref{equa:A+lemma1bound4} with a union bound, it follows that $\mathbb{P}'[\mathcal{E}']=1-e^{-\Omega(\log^2n)}$.
Therefore, $\mathbb{P}[\mathcal{E}']\geq1-2/n^3$\COMMENT{By the law of total probability, we have that \[\mathbb{P}[\mathcal{E}']=\mathbb{P}'[\mathcal{E}']\,\mathbb{P}[\mathcal{E}]+\mathbb{P}[\mathcal{E}'\mid\overline{\mathcal{E}}]\,\mathbb{P}[\overline{\mathcal{E}}]\geq\mathbb{P}'[\mathcal{E}']\,\mathbb{P}[\mathcal{E}]\geq(1-e^{-\Omega(\log^2n)})(1-1/n^3)\geq1-2/n^3\]
(where the last inequality holds for $n$ sufficiently large).}.
Finally, the statement follows by a union bound over all vertices $v\in V$.
\end{proof}

\begin{proof}[Proof of \cref{thm:random}]
Condition on the event that the statements of \cref{lem:betaedges,lem:degrees,coro:excessbound2,lemma:A0edges} hold, which occurs a.a.s.
Then, \cref{lem:betaedges} directly implies \ref{def:digraphitem5} holds.
We may partition the vertices by defining $A^+\coloneqq\{v\in V(D):\ex_D(v)\geq155\kappa\}$, $A^-\coloneqq\{v\in V(D):\ex_D(v)\leq-155\kappa\}$ and $A^0\coloneqq V(D)\setminus(A^+\cup A^-)$.
In particular, by \cref{coro:excessbound2} we have that $|A^0|$ is sublinear.
The condition on the excess in \ref{def:digraphitem1} and \ref{def:digraphitem2} holds now by definition.
The conditions on the edge distribution in \ref{def:digraphitem1} and \ref{def:digraphitem2} as well as \ref{def:digraphitem4} follow by \cref{lemma:A0edges} in the given range of $p$\COMMENT{Let us prove this for one of the cases of \ref{def:digraphitem4} (the other cases are analogous or give slightly weaker bounds).
By \cref{lemma:A0edges}, for each $v\in A^+$ we have that
\[e_D(v,A^-)\geq np/2-2\sqrt{n/(1-p)}\log^2n,\]
so it suffices to check that this is at least $np/3$.
But this is equivalent to having
\[\frac{np}{6}\geq2\sqrt{\frac{n}{1-p}}\log^2n\iff p^2(1-p)\geq144\log^4n/n.\]
This clearly holds for all constant $p\in(0,1)$ (for sufficiently large $n$), so we may assume that $p=o(1)$ or $p=1-o(1)$.
In the first case, the inequality becomes $p^2\geq(1+o(1))144\log^4n/n$, which holds for $p\geq13\log^2n/n^{1/2}$ (and $n$ large enough).
In the second case, we get $(1-p)\geq(1+o(1))144\log^4n/n$, which holds for $p\leq1-150\log^4n/n$ (and $n$ large enough).}.
Finally, \ref{def:digraphitem3} holds by combining \cref{lem:degrees} and \cref{lemma:A0edges}\COMMENT{We have that (we only write one case, the other is analogous)
\[e_D(v,A^0)=d^+_D(v)-e_D(v,A^+)-e_D(v,A^-)\leq np+c\sqrt{np(1-p)\log n} - np + 4\sqrt{n/(1-p)}\log^2n\leq5\sqrt{n/(1-p)}\log^2n\]
(where the last inequality holds for sufficiently large $n$).}.
\end{proof}

\section{Conclusion}

We have shown in \cref{thm:mainintro} that, for $p$ in the range $n^{-1/3}\log^4 n \leq p \leq 1 - n^{-1/5}\log^{5/2}n$, a.a.s.~$D_{n,p}$ is consistent.
Of course, we should expect to be able to improve this range, particularly the lower bound, and perhaps even no lower bound is necessary.
Indeed, when $p \ll 1/n$, we know $D_{n,p}$ is acyclic, and it is easy to see that acyclic digraphs are consistent (simply iteratively remove maximal length paths and observe that the excess decreases by $1$ each time).

The bottleneck in our current approach is in \cref{lem:absorbc2} where we process medium length cycles.
An improvement in the bounds there would lead to an improvement in the range of $p$ in \cref{thm:mainintro}.
However this alone can only achieve a lower bound for $p$ of approximately $n^{-1/2}$: beyond that  one needs to improve other aspects of the argument and new ideas are necessary.

Finally, we saw that our methods can be used to show that a fairly broad class of digraphs (that are far from pseudorandom) are consistent; see \cref{thm:determintro} and \cref{thm:main}.
It would be interesting to find other classes of digraphs that are consistent.

%%%%%%%%%%%%%%%%%%%%%%%%%%%%%%%%%%%%%%%%%%%%%%%%%%%%%%%%%%%%%%%%%%%%%
%%%%%%%%%%%%%%%%%%%%%%%%%%%%%%%%%%%%%%%%%%%%%%%%%%%%%%%%%%%%%%%%%%%%%

\bibliographystyle{mystyle}
\bibliography{references}

\end{document}